\documentclass[12pt]{article}
\usepackage{graphicx} 
\usepackage{times}
\usepackage{tikz-cd}
\usepackage[utf8]{inputenc}
\usepackage{amsmath}
\usepackage{amsfonts}
\usepackage{amssymb}
\usepackage{amsthm}
\usepackage{booktabs}
\usepackage{latexsym}
\usepackage{ytableau}
\usepackage{tikz}
\usepackage[
    backend=biber,
    style=numeric,
  ]{biblatex}
\usepackage{float}
\usepackage{mathdots}
\usepackage[multiple]{footmisc}

\newtheorem{thm}{Theorem}
\newtheorem{prop}[thm]{Proposition}

\newtheorem{defin}[thm]{Definition}
\newtheorem{rmk}[thm]{Remark}

\newtheorem{cor}[thm]{Corollary}

\newtheorem{exa}[thm]{Example}
\newtheorem{notn}[thm]{Notation}

\newtheorem{warn}[thm]{\emph{WARNING}}

\newcommand{\lcm}{\text{lcm}}

\title{Automorphic vector-forms using the Cohn-Elkies magic functions}

\author{Michael Andrew Henry [TU Graz]}
\date{\today}

\begin{document}

\maketitle

\begin{abstract}
    In this study, we introduce the theory of what we call Hecke vector-forms. A Hecke vector-form can be viewed as a vector function representation of some quasiautomorphic form that transforms like an automorphic form on an arbitrarily chosen Hecke triangle group. In other words, because quasiautomorphic forms have complicated transformation behavior when compared with automorphic forms, the construction of a Hecke vector-form is to retrieve a transformation behavior analogous to the simpler, automorphic case. In this way, a Hecke vector-form can be viewed as the vector function analogue of an automorphic form. Since our work is for any quasiautomorphic form over an arbitrary Hecke triangle group, we briefly review the construction of such groups. Furthermore, we review the derivation of the hauptmodul, the automorphic forms, and the normalized quasiautomorphic form of weight $2$ for any Hecke triangle group. We then proceed to the theory of Hecke vector-forms and establish the desired transformation behavior with respect to the generators of the associated group. A proof of this fact is strictly elementary, relying on fine properties of binomial coefficients. Lastly, we relate the vector-forms to Hecke automorphic linear differential equations, which are analogues of the commonly researched modular linear differential equations. Our results include Hecke vector-forms of the classical quasimodular forms as the simplest case. 
    
    \let \thefootnote\relax\footnote{The author would like to thank C. Aistleitner, P. Grabner, H. Movasati, and W. Zudilin for valuable feedback that greatly improved this work. This research has been supported by the Austrian Science Fund (FWF), grant number W1230.}
\end{abstract}


\section{Introduction}

Modular forms are holomorphic functions defined on the upper half plane $ \mathbb{H} $ with an especially elegant transformation behavior for elements of the modular group
    \begin{eqnarray*}
        \mathbf{PSL}_2 (\mathbb{Z}) := \left\{ \frac{az+b}{cz+d}: a,b,c,d \in \mathbb{Z} \text{ and } ad-bc = 1 \right \}.
    \end{eqnarray*}
Functions having this transformation behavior on other subgroups of $ \mathbf{PSL}_2 (\mathbb{R}) $ are called automorphic forms. Here forward, discussion of automorphic forms will include the modular forms--after all, the modular group is indeed the ``first'' Hecke triangle group. 

Automorphic forms are an old study; \emph{quasi}automorphic forms, however, were defined much more recently (in 1995, see \cite{kanekoZagier1995}). The span of time between the introduction of automorphic forms and quasiautomorphic forms might give the impression that these latter functions are only of special interest. It is true, for example, that a quasiautomorphic form does not exhibit the elegant transformation property that is so characteristic of an automorphic form. This makes quasiautomorphic forms more cumbersome in practice, the complication increasing as the depth of the function increases. Nonetheless, it is seen in the breakthrough result \cite{viazovska2017} (similarly, see \cite{cohnEtAl2017} and \cite{feigenbaumEtAl2021}) that the transformation behavior of a quasiautomorphic form is of great power in itself and so here we accept this behavior as basic. That is, quasiautomorphic forms are the Cohn-Elkies ``magic'' functions of the mathematical folklore around sphere packing problems \cite{cohnElkies2003}.

The aim of this article is to introduce what we call the \emph{Hecke vector-forms} (see definition \ref{def:HeckeVectorForm}), which are special vector valued functions over the Hecke triangle groups. In contrast to studies such as \cite{bantayGannon2007} and \cite{francMason2016} that introduce automorphic forms as vector functions more abstractly, we aimed here to begin with quasiautomorphic forms modulo the Hecke triangle groups as the vector components. Since, say, the classical modular group is a Hecke triangle group, these vector functions are in a sense already familiar. We have supplied much detail to provide a cohesive formalism both for working with such functions and also setting the groundwork for a uniform future study. In this spirit, we also prove some results about Hecke vector-forms that have a more or less fundamental feel. These results include a ``corrected'' transformation formula (see Theorem \ref{fundamentalHeckeVector-form}). We also show how vector-forms are deeply connected to a class of automorphic linear differential equations (see Theorem \ref{thm:fundSolutions}). 

Our project was not completed in the greatest generality possible (with respect to what is known about quasiautomorphic forms on Schwarz triangle groups), but instead we chose to highlight the very close formal relationship between automorphic forms on the Hecke triangle groups, generally, and the simplest, most familiar instance, the modular forms on the full modular group. Being more precise, we show that with respect to the formalism we introduce, moving between different Hecke triangle groups is simply a matter of accounting for two parameter values: in our notation, these real parameters are $ C_{\mathfrak{t}_\mu} $ and $ \varpi_\mu $, a structure constant and a shift, respectively. In this way, our study bears heavily on the classical theory of modular forms. As such, we hope the developments will be useful to the usual practitioners from number theory and physics.

One of the (philosophical) beliefs behind this project is that quasiautomorphic forms are just as fundamental as automorphic forms if not the more essential class of objects. Given the priority of automorphic forms in the research literature, going as far back as Riemann and Poincar\'{e}, this claim deserves some justification. Due to space, we will settle for a single, brief explanation by which the reader will at least find something to ponder. We hope to return to this topic again soon, providing more details.

For our argument, we quote a famous theorem by Malmquist from 1913 (see \cite{hille1976} or \cite{steinmetz2017}): 
    \begin{thm}[Malmquist]
        Let $ u(z) $ be defined for $ z \in \mathbb{C} $, $ u'(z) = d \,u(z)/dz$,  and $ R(w) $ be some rational function of $ u $. If $ u(z) $ is a meromorphic function and the ordinary differential equation
            \begin{eqnarray*}
                u' = R(u)
            \end{eqnarray*}
        holds, then $ R $ is a polynomial of degree at most $ 2 $.
    \end{thm}
The reader should compare this result with the basic fact that $ u' = u $ holds for the entire function $ u(z) = e^z $. Roughly, we might say, the exponential function reproduces itself under differentiation. Informally, then, we can interpret Malmquist's theorem as saying, for meromorphic functions there is a hard limit to the extent that it can reproduce itself under differentiation, and this hard limit is as a quadratic polynomial. We add that $ 2 $ cannot be made smaller as the classical so-called Riccati equation makes clear. Sometimes Malmquist's theorem is called the Malmquist-Yosida Theorem, as it was Yosida who showed the theorem has a rich connection with Nevanlinna theory. This connection is at present very actively studied (see e.g. \cite{laine1993} or \cite{steinmetz2017}). 

We relate this now to quasiautomorphic forms. Per the recent work of \cite{doranEtAl2013}, we find that a quasiautomorphic form of weight $ w = 2 $ over a general triangle group can be expressed as a meromorphic function $ u $ (pole at $ i\infty $ or holomorphic at $ i \infty $), satisfying
    \begin{eqnarray}\label{extremalRiccati}
        u' = u^2 + b_1u + b_0, 
    \end{eqnarray}
$ b_1 \neq 0 $ and $b_0 \neq 0 $, which is a case of the aforementioned Riccati differential equation. We can conclude that such quasiautomorphic forms offer us examples that satisfy the ``hard limit'' case of Malmquist's celebrated theorem. For this reason, we suggest quasiautomorphic forms have a hitherto unprescribed naturalness and importance.  
\section{Hecke triangle groups and Hecke quasiautomorphic forms}
In this section, we begin with the theory of the Hecke triangle groups and proceed through until we have what we need to introduce the Hecke vector-forms.

\subsection{Elementary theory of Hecke triangle groups}

Of the general Schwarz triangle groups, the initial triangle on $ \mathbb{C}^* $ (compactified complex plane) is given by the angles $\pi/\mu_1, \pi/\mu_2,  $ and $\pi/\mu_3 $ for $ 2 \leq \mu_1 \leq \mu_2 \leq \mu_3 \leq \infty $ where $ \mu_i \in \mathbb{N}\cup\{\infty\} $. The group is then generated by repeated reflections of the initial triangle, forming a tessellation over the upper half plane (i.e. no gaps and no intersections). There are three types of triangle groups and they are distinguished by which equation
    \begin{eqnarray}\label{triangleCondition}
        \frac{1}{\mu_1} + \frac{1}{\mu_2} + \frac{1}{\mu_3} \,\square\, 1,
    \end{eqnarray}
is satisfied, where $ \square  $ means
    \begin{eqnarray*}
         > \quad \text{or}\quad 
         = \quad \text{or}
        \quad <, 
    \end{eqnarray*} 
respectively. We are only concerned with the final case or the so-called triangle groups of the third kind. The other two kinds of triangle groups are never mentioned again. The triangle groups of the third kind are so-called reflection groups or tessellation groups. Sometimes these are also called Fuchsian groups of the first kind, just as is explained in \cite{sansone1969}, p. 440.  

Standard notation for the general triangle group from complex analysis (e.g. as in \cite{nahari1952}) is according to their so-called signature $ \mathfrak{t} = (\mu_1, \mu_2, \mu_3) $, where $\mu_1 $, $ \mu_2$ and $ \mu_3 $ are the numbers indicated by equation (\ref{triangleCondition}). From here forward we always just assume that $ \mu_3 = \infty $. That is, one angle of the triangle being zero, we naturally choose to fix the vertex at infinity. 

To introduce the Hecke triangle groups, we must specialize the Schwarz triangle groups. A Hecke triangle group is of signature 
    \begin{eqnarray*}
        \mathfrak{t}_\mu : = (2, \mu,\infty)
    \end{eqnarray*}    
where $ 3 \leq \mu \in \mathbb{N} $. 

To present the Hecke triangle groups as we would like, we will introduce a new, more encompassing notation, though it is similar in spirit to that used in \cite{berndtKnopp2008}.

    \begin{notn}
        Critical for our presentation is the fact that the Hecke triangle group $ \mathfrak{t}_\mu $ is generated by the maps
            \begin{eqnarray}\label{generators}
                \begin{cases}
                    S_{\mathfrak{t}_\mu}z := -\frac{1}{z}, \\
                    T_{\mathfrak{t}_{\mu}} z := z + 2 \cos \frac{\pi}{\mu}
                \end{cases}
            \end{eqnarray}
        under composition when $ z \in \mathbb{H} $ (e.g. see \cite{beardon1983}). 
    \end{notn}
       
    \begin{notn}
        From the shift in generator $ T_{\mathfrak{t}_\mu} z $, we fix
            \begin{equation}\label{HeckeShiftTerm}
                \varpi_\mu = \varpi : =  2 \cos \frac{\pi}{\mu}
            \end{equation}
        where $ 2 < \mu \in \mathbb{N} $.
    \end{notn} 
Once $ \mu $ is chosen, we want a more convenient normalization of the triangle groups for use with quasiautomorphic forms. Toward this end for a triangle group $(\mu_1, \mu_2, \infty ) $ let the corner vertices $ r_1, r_2, $ and $ r_3 $ of the initial curvilinear triangle satisfy
    \begin{eqnarray}\label{normalizedVertices}
        \begin{cases}
            r_1 := -{e^{-\pi i/\mu_1 }}, \\
            r_2 := -e^ {-\pi i/\mu_2}, \\
            r_3 := \infty;
        \end{cases}
    \end{eqnarray}
this is the classical normalization of $ r_1 $, $ r_2 $ and $ r_3 $ and we have no need to alter it. Our definitions later of quasiautomorphic forms will depend on this normalization by a factor of some root of unity; we have normalized in order to fix the factor simply as $ 1 $.  

\begin{defin}
    Let $ 3 \leq \mu \in \mathbb{N}$. We denote the normalized $ (\mu-2) $-th \emph{Hecke triangle group} by
        \begin{eqnarray*}
            \mathfrak{H}(\varpi_\mu) := (2,\mu,\infty)
        \end{eqnarray*}
    where $ \varpi_\mu $ is given by \eqref{HeckeShiftTerm}, the initial triangle of $ \mathfrak{t}_\mu $ has vertices $ r_1 = i $, $ r_2 = -e^{-\pi i/ \mu} $ and $ r_3 = \infty $.
\end{defin}

To find the so-called fundamental domain of $\mathfrak{t}_\mu $, we reflect the normalized initial triangle $ d $ (from \eqref{normalizedVertices}, etc.) once across the axis of symmetry, denoted $ \alpha $ in Figure \ref{curvTiangle}. The interior if this domain, removing cusps, is the fundamental domain. 
   \begin{figure}[H]
        \centering
        \begin{tikzpicture}[z=-1cm,-,thick, scale=2]
            \coordinate [label=below:\textcolor{black}{$r_1$}]  (A) at (0,1);
            \coordinate [label=below:\textcolor{black}{$r_2 $}]  (B) at (-.5,0.866);
            \coordinate [label=left:\textcolor{black}{$r_3$}]  (C) at (-0.5,2);
            \coordinate [label=below:\textcolor{black}{$0$}]  (D) at (0,0);
            \coordinate [label=below:\textcolor{black}{$-1$}]  (D) at (-1,0);
            \coordinate [label=below:\textcolor{black}{$1$}]  (D) at (1,0);
            \coordinate [label=right:\textcolor{black}{$ \alpha $}]  (G) at (0,1.7);
            \coordinate [label=right:\textcolor{black}{$ d $}]  (G) at (-0.42, 1.5);
            \draw[color=black] (-0.5,.866) -- (-0.5,2.3);
            \draw[color=gray] (0,1) -- (0,2.3) ;
            \draw[color=black] (0.5, 0.866) -- (0.5,2.3);
            \draw[ color=black] (-1.5,0,0) -- (1.5,0,0);
            \draw[] (-1,0) .. controls (-1,0.555) and (-0.555,1) .. (0,1)
            .. controls (0.555,1) and (1,0.555) .. (1,0);
        \end{tikzpicture}
        \caption{Initial triangle $ d $ with vertices $r_1$, $r_2 $, and $r_3$, and axis of symmetry, $ \alpha $}\label{curvTiangle}
    \end{figure}
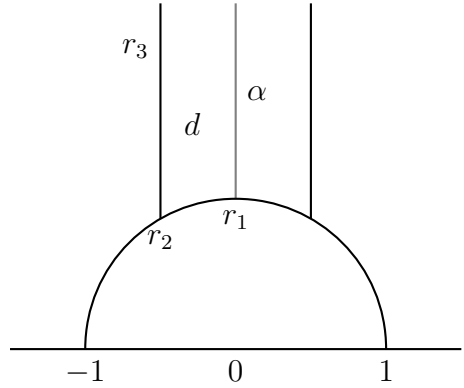
Similarly to how we consider a function like, say, $\sin x $ just over a single period, guaranteeing an inverse, we can restrict the domain of an automorphic form to the fundamental domain. 
    \begin{exa}
        Observe that for $ \mu = 3 $, we have $ \varpi_{\mu} = 2 \cos (\pi/3) = 1 $ or $ T_{\mathfrak{t}_3}z = z+1 $, which is exactly what we should expect for the full modular group, $ \mathfrak{H}(1) $. 
    \end{exa}
In Table \ref{tab:heckeData}, we give a summary of features for the first several Hecke triangle groups.
    \begin{table}[H]
        \centering
        \begin{tabular}{c|c|c|c|c}
            $ \mathfrak{t}_\mu  $ & $ \varpi_{\mu} $ & $ r_1 $ & $ r_2 $ & $ r_3 $  \\
             \hline
            $ (2, 3, \infty) $ & $1$ & $i$ & $ -\frac{1}{2}+ i\frac{\sqrt{3}}{2} $ & $ \infty$ \\
            $ (2, 4, \infty) $ & $ \sqrt{2} $  & $i$ & $ \frac{-\sqrt{2}}{2} + i\frac{\sqrt{2}}{2} $ & $  \infty$ \\
            $ (2 , 5, \infty) $ & $ \frac{1}{2}+\frac{\sqrt{5}}{2} $ & $ i $ & $-\frac{1+\sqrt{5}}{4} + i\sqrt{\frac{5}{8} - \frac{\sqrt{5}}{8} } $ & $ \infty$ \\
            $ (2,6, \infty) $ & $ \sqrt{3} $ & $i$ & $ -\frac{\sqrt{3}}{2} + i\frac{1}{2} $ & $ \infty $  \\
        \end{tabular} 
        \caption{Data for first few initial triangles of $\mathfrak{H}(\varpi_\mu)$}
        \label{tab:heckeData}
    \end{table}
We want to point out that for arbitrary $ 3 \leq \mu \in \mathbb{N} $, the inequality $ 1 \leq \varpi_\mu < 2 $ holds, since $ 2 \cos(\pi/3) = 1 $ and $ 2\cos(\pi/\mu) \rightarrow  2 $ as $\mu \rightarrow \infty. $ The bound on $ \varpi_{} $ matters later for dimension considerations  of the vector space of quasiautmorphic forms defined over $ \mathfrak{H}(\varpi_\mu) $. For all the details see \cite{berndtKnopp2008}, ch. 4.

\subsection{Darboux-Halphen differential systems and the Hecke hauptmodul}

We move toward the general construction of quasiautomorphic forms over an arbitrary Hecke triangle group $\mathfrak{H} (\varpi_\mu) $. We will achieve the needed functions by use of the so-called generalized Darboux-Halphen system. For an alternative development of modular forms see \cite{sebbar2016}.

We assume from this point forward that $ \mathfrak{t}_\mu = \mathfrak{H}(\varpi_\mu) $ (i.e. $\mathfrak{t}_\mu $ is normalized). This particular process of finding quasiautomorphic forms amounts to finding solutions of a special differential system that is very much of interest in its own right. These solutions are necessarily quasiautomorphic forms (as shown in Theorem 4 (i) of \cite{doranEtAl2013}); however, basic operations on the solutions give us the so-called hauptmodul, the so-called automorphic discriminant, and the automorphic forms analogous to Eisenstein series. We will also find explicitly from the differential solutions the unique weight $ 2 $ normalized quasiautomorphic form. This section comes adapted from \cite{doranEtAl2013}, basically with only small changes in notation. 

We start our explanation with a more general (and more symmetric) differential system (see \cite{doranEtAl2013}, Formula (1.2) there), namely
    \begin{defin}
        Consider a Schwarz triangle group given by signature $ \mathfrak{t} = (\mu_1, \mu_2, \infty) $. Let $ y_1(q) $, $ y_2(q) $, and $ y_3(q) $ be complex functions for $ q = q_{\mathfrak{t}_\mu}:= e^{2\pi i z /\varpi_\mu}$ where $ z \in \mathbb{H} $. We define three constants
            \begin{eqnarray}\label{schwarzConstants}
                \begin{cases}
                    a_\mathfrak{t} := \frac{1}{2}\left(1 - \frac{1}{\mu_1} + \frac{1}{\mu_2} \right), \\
                    b_{\mathfrak{t}} := \frac{1}{2}\left(1 - \frac{1}{\mu_1} - \frac{1}{\mu_2} \right), \\
                    c_\mathfrak{t} := \frac{1}{2}\left(1 + \frac{1}{\mu_1} - \frac{1}{\mu_2} \right).
                \end{cases}
            \end{eqnarray}
        If $ \theta f := q \, df(q)/dq = \varpi(2\pi i) ^{-1} {d f(z)}/{dz} $, then the differential system
            \begin{eqnarray}\label{darbouxHalphen}
                &{}& \mathsf{DH}_\mathfrak{t}[y_{\mathfrak{t},1},y_{\mathfrak{t},2},y_{\mathfrak{t},3}] :\\
                &{}&
                    \begin{cases}
                        \theta y_{\mathfrak{t},1} := (a_\mathfrak{t}-1)(y_{\mathfrak{t}, 1}y_{\mathfrak{t},2} + y_{\mathfrak{t},1}y_{\mathfrak{t},3}-y_{\mathfrak{t},2}y_{\mathfrak{t},3}) + (b_{\mathfrak{t}}+c_\mathfrak{t}-1)y_{\mathfrak{t},1}^2, \\
                        \theta y_{\mathfrak{t},2} := (b_\mathfrak{t}-1)(y_{\mathfrak{t},2}y_{\mathfrak{t},1}+ y_{\mathfrak{t},2}y_{\mathfrak{t},3} - y_{\mathfrak{t},1} y_{\mathfrak{t},3}) + (a_\mathfrak{t}+c_\mathfrak{t}-1)y_{\mathfrak{t},2}^2, \\
                        \theta y_{\mathfrak{t},3} := (c_\mathfrak{t}-1)(y_{\mathfrak{t},3}y_{\mathfrak{t},1} + y_{\mathfrak{t},3} y_{\mathfrak{t},2} - y_{\mathfrak{t},1}y_{\mathfrak{t},2}) + (a_\mathfrak{t}+b_\mathfrak{t}-1)y_{\mathfrak{t},3}^2
                    \end{cases}
            \end{eqnarray}
        is a so-called \emph{generalized Darboux-Halphen system}.
    \end{defin}

We next specialize the system at (\ref{schwarzConstants}) and (\ref{darbouxHalphen}) to deal with only the Hecke triangle groups $ {\mathfrak{t}_\mu = (2,\mu,\infty)} $. The change in notation in the proceeding differential system has the obvious meaning. 

    \begin{defin}
        Again define $q = q_{\mathfrak{t}_\mu}:= e^{2 \pi i z /\varpi_\mu} $. Let the Hecke triangle group $\mathfrak{H}_\mu (\varpi)$ be given. If
            \begin{eqnarray}\label{darbouxHalphenConstants}
                \begin{cases}
                    a_{\mathfrak{t}_\mu} := \frac{1}{2}\left(\frac{1}{2} + \frac{1}{\mu} \right) = \frac{\mu + 2}{4 \mu}, \\
                    b_{\mathfrak{t}_\mu} := \frac{1}{2}\left(\frac{1}{2} - \frac{1}{\mu} \right)= \frac{\mu - 2}{4 \mu}, \\
                    c_{\mathfrak{t}_\mu} := \frac{1}{2}\left(\frac{3}{2} - \frac{1}{\mu} \right) = \frac{3\mu - 2}{4 \mu},
                \end{cases}
            \end{eqnarray}
        then the \emph{Hecke Darboux-Halphen system} for $ \mathfrak{H}( \varpi_\mu) $ is 
            \begin{eqnarray}\label{halphenForHeckeGeneral}
                &{}&\mathsf{DH}_{\mathfrak{t}_\mu}[y_{\mathfrak{t}_\mu,1},y_{\mathfrak{t}_\mu,2},y_{\mathfrak{t}_\mu,3}] : \\
                    &{}&
                    \begin{cases}
                        \theta y_{\mathfrak{t}_\mu, 1} = - \frac{3\mu-2}{4\mu}  (y_{\mathfrak{t}_\mu, 1}y_{\mathfrak{t}_\mu,2} + y_{\mathfrak{t}_\mu,1}y_{\mathfrak{t}_\mu,3}-y_{\mathfrak{t}_\mu,2}y_{\mathfrak{t}_\mu,3}) -\frac{1}{\mu}y_{\mathfrak{t}_\mu,1}^2,
                        \\
                        \theta y_{\mathfrak{t}_\mu, 2} = - \frac{3\mu+2}{4\mu} (y_{\mathfrak{t}_\mu,2}y_{\mathfrak{t}_\mu,1}+ y_{\mathfrak{t}_\mu,2}y_{\mathfrak{t}_\mu, 3} - y_{\mathfrak{t}_\mu,1} y_{\mathfrak{t}_\mu,3}),
                        \\
                        \theta y_{\mathfrak{t}_\mu, 3} = - \frac{\mu-2}{4\mu} (y_{\mathfrak{t}_\mu, 3}y_{\mathfrak{t}_\mu,1} + y_{\mathfrak{t}_\mu,3} y_{\mathfrak{t}_\mu, 2} - y_{\mathfrak{t}_\mu,1}y_{\mathfrak{t}_\mu,2}) -\frac{1}{2}y_{\mathfrak{t}_\mu,3}^2 .
                    \end{cases}  
            \end{eqnarray}
    \end{defin}

We now come to the construction of the functions we require on $ \mathfrak{H}(\varpi_\mu) $. To start, let us define a linear combination of differential solutions from $ \mathsf{DH}_{\mathfrak{t}_\mu } $ at (\ref{halphenForHeckeGeneral}), using
    \begin{notn}\label{solutionDifferenceNotation}
        Let $ q = q_{\mathfrak{t}_\mu} := e^{2 \pi i z/\varpi_\mu} $. Let $z \in \mathbb{H}$ and $ y_{\mathfrak{t}_\mu, j }$ for $ 1\leq j \leq 3 $ be solutions of the differential system at \eqref{halphenForHeckeGeneral}. We define
            \begin{eqnarray}
                \begin{cases}
                    \mathsf{u}_{\mathfrak{t}_\mu} := y_{\mathfrak{t}_\mu, 1}(q) - y_{\mathfrak{t}_\mu, 2}(q), \\
                    \mathsf{v}_{\mathfrak{t}_\mu} := y_{\mathfrak{t}_\mu, 3}(q) - y_{\mathfrak{t}_\mu, 2}(q).
                \end{cases}
            \end{eqnarray}    
    \end{notn}

    \begin{defin}[Hecke automorphic function]
         For $ z \in \mathbb{H}$,  $ f(z) $ is a \emph{Hecke automorphic function} on $ \mathfrak{H}(\varpi_\mu) $ if whenever
            \begin{eqnarray*}
                \gamma_{\mathfrak{t}_\mu}(z) = \frac{az+b}{cz+d} \in \mathfrak{H}(\varpi_\mu),
            \end{eqnarray*}
        holds, then
            \begin{eqnarray*}
                f(\gamma_{\mathfrak{t}_\mu}(z)) = f(z).
            \end{eqnarray*}
        holds. 
    \end{defin}

Everything we will need can be described in terms of basic operations on $ \mathsf{u}_{\mathfrak{t}_\mu} $ and $ \mathsf{v}_{\mathfrak{t}_\mu} $ of $
\mathsf{DH}_{\mathfrak{t}_\mu} $. We begin by constructing what is analogous to the classical hauptmodul for the general Hecke triangle group $ \mathfrak{H}(\varpi_\mu) $ that we denote by $ J_{\mathfrak{t}_\mu} $. Naturally, the classical hauptmodul $ J = j/1728 $ for the modular group we would denote as $ J_{\mathfrak{t}_3} $, etc. We have
    \begin{defin}\label{hauptmodulFromHalphen}
        The \emph{Hecke hauptmodul} on $\mathfrak{H}(\varpi_\mu) $ we define by
            \begin{eqnarray*}
                 J_{\mathfrak{t}_\mu} : = \frac{y_{\mathfrak{t}_\mu, 3} - y_{\mathfrak{t}_\mu, 2}}{y_{\mathfrak{t}_\mu, 3} - y_{\mathfrak{t}_\mu, 1}} = \frac{\mathsf{v}_{\mathfrak{t}_\mu}}{\mathsf{v}_{\mathfrak{t}_\mu}- \mathsf{u}_{\mathfrak{t}_\mu}}.  
            \end{eqnarray*}
    \end{defin}
    \begin{thm}
        The Hecke hauptmodul $J_{\mathfrak{t}_\mu}  $ is an automorphic function. 
    \end{thm}
    \begin{proof}
        Taken from \cite{doranEtAl2013}.
    \end{proof}

Next, using the appropriate M\"{o}bius map, we can always ensure that the conditions 
    \begin{eqnarray*}
        \begin{cases}
        J_{\mathfrak{t}_\mu} (r_1) = 1, \\
        J_{\mathfrak{t}_\mu} (r_2) = 0,  \\
        J_{\mathfrak{t}_\mu} (r_3) = \infty
        \end{cases}
    \end{eqnarray*}
are satisfied.  This is completely in accord with the familiar, classical situation where $ J_{\mathfrak{t}_3}\left(i\right) = 1 $, $ J_{\mathfrak{t}_3}\left(\frac{1+i\sqrt{3}}{2}\right) = 0 $, and $  J_{\mathfrak{t}_3}(i\infty) = \infty $. Looking at Figure \ref{curvTiangle}, the initial triangle $ d $ is mapped conformally by $ J_{\mathfrak{t}_\mu}(z) $ to $\mathbb{H}$ except at the cusps where it is not conformal, and the reflection of $ d $ over the axis $ \alpha $ is mapped by $ J_{\mathfrak{t}_\mu}(z) $ to the mirror image of $ \mathbb{H} $ (i.e. the lower half plane). For more details, the reader can see \cite{apostol1976II}.

Briefly, as it was not mentioned in \cite{doranEtAl2013}, observe that directly from Definition \ref{hauptmodulFromHalphen}, we can determine that, if $ \gamma_{\mathfrak{t}_\mu}(z) \in \mathfrak{H}(\varpi_\mu) $, then 
    \begin{eqnarray*}
        \frac{\gamma_{\mathfrak{t}_\mu}(y_{\mathfrak{t}_\mu, 3}) - \gamma_{\mathfrak{t}_\mu}(y_{\mathfrak{t}_\mu, 2}) }{\gamma_{\mathfrak{t}_\mu}(y_{\mathfrak{t}_\mu, 3}) - \gamma_{\mathfrak{t}_\mu}(y_{\mathfrak{t}_\mu, 1})} = \frac{y_{\mathfrak{t}_\mu, 3} - y_{\mathfrak{t}_\mu, 2}}{y_{\mathfrak{t}_\mu, 3} - y_{\mathfrak{t}_\mu, 1}} 
    \end{eqnarray*}
holds. The affine-ratio (\cite{busemannKelly1953}, p. 30; \cite{casas-alvero2014}, p. 57) from the theory of projective geometry, as appears in Definition \ref{hauptmodulFromHalphen}, is invariant under M\"{o}bius transformations. The affine-ratio is a specialization of the fundamental invariant of projective geometry, the cross-ratio, by taking one of the four parameters as $ \infty $. We quickly review the relevant concepts from projective geometry.
    \begin{defin}
        Let $ z_1 $,$ z_2 $,$ z_3 $ and $z_4$ be any four points on the compactified complex plane $ \mathbb{C} $. Then we define the \emph{cross-ratio} as 
            \begin{eqnarray*}
                [z_1, z_2 , z_3 , z_4] : = \frac{(z_1-z_3)(z_2-z_4)}{(z_1-z_2)(z_3 - z_4)}.
            \end{eqnarray*}
    \end{defin}
Note that the cross-ratio goes by many names, such as anharmonic-ratio, double-ratio, etc.  If $ z_4 \rightarrow \infty $, then by using the cross-ratio and continuity, it follows that we can sensibly write
    \begin{eqnarray*}
        [z_1, z_2 , z_3 , \infty] = \frac{ z_1-z_3 }{ z_1-z_2}.
    \end{eqnarray*}
holds. We codify this in
    \begin{defin}\label{affine-ratio}
        The so-called \emph{affine-ratio} is defined by
            \begin{eqnarray}
                [z_1, z_2, z_3] := [z_1, z_2 , z_3 , \infty] 
            \end{eqnarray}
    \end{defin}
        
This completes our introduction of the Hecke hauptmodul and the brief review of projective geometry. For a complete reference on the latter subject see \cite{casas-alvero2014}.

\subsection{Darboux-Halphen also gives (quasi)automorphic forms}
In the previous section we used the Darboux-Halphen system to derive the function $J_{\mathfrak{t}_\mu} $, an automorphic function. Next, we want to use the Darboux-Halphen system to derive the automorphic forms. 
    \begin{defin}[Hecke automorphic form]\label{def:automForm}
        Let
            \begin{eqnarray*}
                \gamma_{\mathfrak{t}_\mu}(z) = \frac{az+b}{cz+d} \in \mathfrak{H}(\varpi_\mu) .
            \end{eqnarray*}
        Then if the relation
            \begin{eqnarray*}
                \frac{f\left(\gamma_{\mathfrak{t}_\mu}(z)\right)}{(cz+d)^{2k} } = f(z)
            \end{eqnarray*}
        holds for $ z\in \mathbb{H} $, then we say $ f $ is a \emph{Hecke automorphic form} of weight $ 2k $ on $\mathfrak{H}(\varpi_\mu)$.
    \end{defin}
We aim to produce the special automorphic forms on $ \mathfrak{H}(\varpi_\mu) $ that are the direct analogues of the classical Eisenstein series already familiar from $ \mathfrak{H}(1) $. First we must give

    \begin{defin} 
        For any $ 2 \leq k \in \mathbb{N} $, the \emph{Eisenstein series} of weight $2k$ with respect to $\mathfrak{H}(\varpi_\mu) $ is defined by
            \begin{eqnarray}\label{EisensteinFromHalphen}
                E_{ \mathfrak{t}_\mu, 2k}(z) := \left(\frac{\varpi_\mu }{2 \pi i}\right)^k \mathsf{u}_{\mathfrak{t}_\mu} \mathsf{v}_{\mathfrak{t}_\mu}^{k-1}.
            \end{eqnarray}
    \end{defin}
    \begin{thm} 
        The Eisenstein series $ E_{\mathfrak{t}_\mu, 2k} $ is an automorphic form of weight $ 2k $.
    \end{thm}
    \begin{proof}
        Taken from \cite{doranEtAl2013}.
    \end{proof}
    
    \begin{notn}
        The vector space of Hecke automorphic forms of weight $ w = 2k $ for $2 \leq k \in \mathbb{Z} $ over $ \mathfrak{H}(\varpi_\mu) $ we denote by
            \begin{eqnarray*}
               \mathcal{M}_w(\mathfrak{H}(\varpi_\mu)). 
            \end{eqnarray*}
         For a generic function in $ \mathcal{M}_{w}(\mathfrak{H}(\varpi_\mu)) $ we often write $ B_{\mathfrak{t}_\mu, w}(z) $.
    \end{notn}

The following Theorem giving the dimension formula is non-trivial and taken from \cite{berndtKnopp2008}, p. 52. In the source, their Theorem is stated in greater generality than we need due to our normalization of the initial triangle of $ \mathfrak{H}(\varpi_\mu) $. We state their theorem in this more specialized form. 
    \begin{thm}\label{thm:HeckeModularDimension}
        Let $ 3 \leq \mu \in \mathbb{Z} $. Let $ w = 4k $ be the weight of a function in the space $ \mathcal{M}_{w}(\mathfrak{H}(\varpi_\mu)) $ and $ \lfloor x \rfloor $ is the so-called floor function. Then the relation
            \begin{eqnarray*}
                \dim \mathcal{M}_{w}(\mathfrak{H}(\varpi_\mu))  =  \left\lfloor \frac{w(\mu-2)}{4\mu} \right\rfloor + 1
            \end{eqnarray*}
        holds. 
    \end{thm}
    \begin{proof}
        See \cite{berndtKnopp2008}, ch. 5, p.52. 
    \end{proof}

Having already used the Darboux-Halphen sytem to find $ {J}_{\mathfrak{t}_\mu} $, we also can recover a function analogous to the so-called modular discriminant. 
    \begin{defin}[Hecke automorphic discriminant]\label{Heckediscriminant}
        For even $ \mu $, say $ \mu = 2\mu_0 $, we find that $  \emph{\lcm}(2,\mu) = \mu $, and consequently we define the \emph{Hecke automorphic discriminant} by
            \begin{eqnarray*}
                 \Delta_{\mathfrak{t}_\mu} :=  \left( \frac{ \theta J_{\mathfrak{t}_\mu}}{J_{\mathfrak{t}_\mu}(J_{\mathfrak{t}_\mu}-1)}\right)^{\mu} J _{\mathfrak{t}_\mu} (J_{\mathfrak{t}_\mu}-1)^{\mu_0}  ;
            \end{eqnarray*}
        however, if $ \mu $ is odd, then $ \emph{\lcm}(2,\mu) = 2 \mu $ and hence
            \begin{eqnarray*}
                 \Delta_{\mathfrak{t}_\mu} := \left( \frac{ \theta J_{\mathfrak{t}_\mu}}{J_{\mathfrak{t}_\mu}(J_{\mathfrak{t}_\mu}-1)}\right)^{2\mu} J^2 _{\mathfrak{t}_\mu} (J_{\mathfrak{t}_\mu}-1)^{\mu}.
            \end{eqnarray*}
    \end{defin}

    \begin{thm}
        The Hecke discriminant function $ \Delta_{\mathfrak{t}_{\mu}} $ is a weight $ w = 2 \, \emph{\lcm} (2,\mu) $ automorphic form over $ \mathfrak{H}(\varpi_\mu) $. 
    \end{thm}
    \begin{proof}
        Taken from \cite{doranEtAl2013}.
    \end{proof}
    \begin{notn}\label{discriminantWeight}
        From the previous theorem, it will be convenient to fix the notation
            \begin{eqnarray*}
                \delta_{\mathfrak{t}_\mu} : = 2
                \, \emph{\lcm}(2,\mu),   
            \end{eqnarray*}
        as it is a parameter that occurs frequently.
    \end{notn}

    \begin{defin}[Hecke quasiautomorphic form]   
        Let $ E_{\mathfrak{t}_\mu, 2}(i
        \infty) = C_{\mathfrak{t}_\mu} $ and let  
            \begin{eqnarray*}
                \gamma_{\mathfrak{t}_\mu}(z) = \frac{az+b}{cz+d} \in \mathfrak{H}(\varpi_\mu) .
            \end{eqnarray*}
        If for some $ f(z) $ where $ z
        \in \mathbb{H} $ satisfies 
            \begin{equation*}   
                \frac{f\left( \gamma_{\mathfrak{t}_3}z\right)}{(cz+d)^2} =    f(z) + C_{\mathfrak{t}_\mu}\, \frac{c}{ (cz+d)},
            \end{equation*}
        then we say $f$ is a \emph{Hecke quasiautomorphic form} weight $2$ and of depth $ r = 1$.
    \end{defin}
We note that from $ E_{\mathfrak{t}_\mu, 2} $ all other quasiautomorphic forms are forthcoming when we introduce depth $ r \in \mathbb{N} $ e.g. see Definition \ref{deepVectorSpace}.

Quasiautomorphic forms for $ \mathfrak{H}(1) $ were introduced in \cite{kanekoZagier1995} as a generalization of modular forms. A self-contained introduction to these functions can be found in \cite{royer2012}.
    \begin{defin}
        We define
            \begin{eqnarray*}
               E_{\mathfrak{t}_\mu, 2}(z) := \frac{ C_{\mathfrak{t}_\mu}}{b_{\mathfrak{t}_\mu}}
               \left(({b_{\mathfrak{t}_\mu}-a_{\mathfrak{t}_\mu}})y_{\mathfrak{t}_\mu, 1} - {b_{\mathfrak{t}_\mu}}y_{\mathfrak{t}_\mu, 2} + (b_{\mathfrak{t}_\mu}+{a_{\mathfrak{t}_\mu}-1})y_{\mathfrak{t}_\mu, 3} \right)  
            \end{eqnarray*} 
        where $ C_{\mathfrak{t_\mu}} $ is the structure constant found at Definition \ref{constantOfVanishing} and the constants $a_{\mathfrak{t}_\mu}$ and $b_{\mathfrak{t}_\mu}$ come from \eqref{darbouxHalphenConstants}.
    \end{defin}

    \begin{thm}
        The function $ E_{\mathfrak{t}_\mu, 2} $ is the unique normalized quasiautomorphic form of weight $ 2 $ on $ \mathfrak{H}(\varpi_\mu) $.
    \end{thm}
    \begin{proof}
        Taken from \cite{doranEtAl2013}.
    \end{proof}
The unique quasiautomorphic form $ E_{\mathfrak{t}_\mu,2 } $ is perhaps the most important quasiautomorphic form. Connected to this function is a structure constant. We present this now in
    \begin{defin}[Structure constant]\label{constantOfVanishing}
        We define the \emph{structure constant} by
            \begin{eqnarray}
                C_{\mathfrak{t}_\mu} := E_{\mathfrak{t}_\mu, 2}(i
                \infty) =\frac{\delta_{\mathfrak{t}_\mu}}{2\pi i} = 
                    \begin{cases}
                        \frac{\mu }{\pi i} \qquad \mu \text{ even},\\
                        \frac{2\mu }{\pi i} \qquad \mu \text{ odd}.
                    \end{cases}
            \end{eqnarray}
    \end{defin}

    \begin{exa}
        Notice that when $  S_{\mathfrak{t}_3}z = -1/z $ and $ \mu = 3 $, then we have $ C_{\mathfrak{t}_3} = 6/ \pi i $ and
            \begin{equation*}   
                \frac{E_{\mathfrak{t}_3, 2}\left( S_{\mathfrak{t}_3}z\right)}{z^2} =    E_{\mathfrak{t}_3,2}(z) - C_{\mathfrak{t}_\mu} S_{\mathfrak{t}_\mu}z
            \end{equation*}
        holds, a formula we are already familiar with. 
    \end{exa}

By a direct application of results in \cite{doranEtAl2013} or also \cite{ashokEtAl2020II}, we can give  explicitly the more general ``Ramanujan differential system'' for automorphic forms on $ \mathfrak{H}(\varpi_\mu) $. 
    \begin{cor}[Ramanujan differential system]\label{generalizedRamanjanSystem}
         Our functions $ f(z) $ all have Fourier expansions in terms of $ q = q_{\mathfrak{t}_\mu} := e^{2 \pi i z/\varpi_\mu } $ or $ g(q) $. Let $ \theta f(z) = q \, d  g(q)/ dq $ i.e $ \theta $ is the so-called Boole differential operator. Then the differential system
            \begin{eqnarray*}
                    \begin{cases}
                        \theta E_{\mathfrak{t}_\mu, 2} \, = 1\,\frac{\mu-2}{4\mu} E_{\mathfrak{t}_\mu,2}E_{\mathfrak{t}_\mu, 2}  - \frac{\mu-2}{4\mu}E_{\mathfrak{t}_\mu,4}
                        , \\
                        \theta E_{\mathfrak{t}_\mu, 4} \,= 2\,\frac{\mu-2}{2\mu}E_{\mathfrak{t}_\mu,4}E_{\mathfrak{t}_\mu,2}  - \frac{\mu-2}{\mu}E_{\mathfrak{t}_\mu,6} - \frac{2-2}{2}E_{\mathfrak{t}_\mu,4}E_{\mathfrak{t}_\mu,2 }, \qquad (\dagger) \\
                        \theta E_{\mathfrak{t}_\mu,6} \,= 3\,\frac{\mu-2}{2\mu} E_{\mathfrak{t}_\mu,6}E_{\mathfrak{t}_\mu,2} -\frac{\mu-3}{\mu}E_{\mathfrak{t}_\mu,8} - \frac{3-2}{2}E_{\mathfrak{t}_\mu,4}E_{\mathfrak{t}_\mu,4} , \\
                        \theta E_{\mathfrak{t}_\mu,8} \,= 4\,\frac{\mu-2}{2\mu}E_{\mathfrak{t}_\mu,8}E_{\mathfrak{t}_\mu,2} -\frac{\mu-4}{\mu}E_{\mathfrak{t}_\mu,10} - \frac{4 - 2}{2}E_{\mathfrak{t}_\mu,4}E_{\mathfrak{t}_\mu,6}, \\
                        \qquad \vdots   \\
                        \theta E_{\mathfrak{t}_\mu, 2\mu} = \mu\, \frac{\mu-2}{2\mu}E_{\mathfrak{t}_\mu, 2\mu}E_{\mathfrak{t}_\mu,2} -\frac{\mu-\mu}{\mu}E_{\mathfrak{t}_\mu, 2\mu+2}-\frac{\mu - 2}{2}E_{\mathfrak{t}_\mu,4} E_{\mathfrak{t}_\mu, 2\mu-2}\qquad (\dagger)
                    \end{cases}    
            \end{eqnarray*}
        holds.
    \end{cor}
    \begin{proof}
        This is direct substitution of the Hecke triangle data in Lemma 2.1 found in \cite{ashokEtAl2020I}. Or, again after substitution, the same result can be derived from Lemma 2.5 from \cite{ashokEtAl2020II}.
    \end{proof}

    \begin{warn}
        Observe that there is some redundancy in the presentation of Corollary \ref{generalizedRamanjanSystem} to make the pattern more clear. We have indicated where there are occurrences of vanishing with $ (\dagger) $.  
    \end{warn}

    \begin{exa}\label{ramanujanSystem}
        One import of Corollary \ref{generalizedRamanjanSystem} for quasiautomorphic forms over $ \mathfrak{H}(1) $ is that the classical so-called Ramanujan differential system, namely
            \begin{equation*}
                \begin{cases}
                    \theta E_2 = \frac{1}{12}\left(E_2 E_2 - E_4\right), \\
                    \theta E_4 = \frac{1}{3}\left( E_4 E_2 - E_6\right), \\
                    \theta E_6 = \frac{1}{2}\left( E_6  E_2 - E_4^2\right),
                \end{cases}   
            \end{equation*}
        is a special instance. This Ramanujan differential system has been rediscovered many times (as partly recounted in \cite{moralesEtAl}). See \cite{zudilin2000} for analogues of the Ramanujan system in the special context of  automorphic forms on subgroups of $ \mathfrak{H}(1) $. Related systems can also be found in \cite{hahn2008},  \cite{huberLara2011}, \cite{nikdelan2022}, \cite{ohyama1995}, and \cite{ohyama1996}.
    \end{exa}

Table \ref{tab:functionData} quickly summarizes the functions that we have defined for general Hecke triangle groups that are completely analogous to the classical situation.

\begin{table}[H]
            \centering
            \begin{tabular}{c|c|c}
                function & a so-called & weight  \\ \hline 
                 $J_{\mathfrak{t}_\mu} $ & automorphic function & $0$  \\ \hline
                 $ E_{\mathfrak{t}_\mu,2} $ & quasiautomorphic form & $2$ \\ \hline
                 $E_{\mathfrak{t}_\mu,4} $ & automorphic form & $4$ \\
                 $E_{\mathfrak{t}_\mu, 6} $ & automorphic form & $6$ \\
                 $\vdots $ & $\vdots $ & $\vdots $ \\
                 $E_{\mathfrak{t}_\mu, 2\mu }$ & automorphic form & $2 \mu$ \\
                 $ \Delta_{\mathfrak{t}_\mu} $ & automorphic form &  $2 \lcm(2, \mu) = \delta_{\mathfrak{t}_\mu}$ 
            \end{tabular}
            \caption{Function data modulo $\mathfrak{H}(\varpi_\mu) $}
            \label{tab:functionData}
        \end{table}

\subsection{Vector spaces of Hecke quasiautomorphic forms and the Ramanujan-Serre operator}
The following is an analogue of the standard definition for quasimodular forms. 
    \begin{defin}\label{deepVectorSpace}
        The vector space of quasiautomorhic forms of weight $ w $ and of depth $ r \in \mathbb{N}\cup \{0\} $ is defined by
            \begin{eqnarray*}
                \mathcal{QM}^{r}_w(\mathfrak{H}(\varpi_\mu)) := \bigoplus_{m=0}^r \mathcal{M}_{w-2k}(\mathfrak{H}(\varpi_\mu))E_{\mathfrak{t}_\mu, 2}^m 
            \end{eqnarray*}
    \end{defin}

Observe that this allows us to write any quasiautomorphic form $ U_{\mathfrak{t}_\mu, w, r} \in \mathcal{QM}^{r}_w(\mathfrak{H}(\varpi_\mu))$ as 
    \begin{eqnarray*}
        U_{\mathfrak{t}_\mu, w, r}(z) = 
        \sum_{k=0}^{r} B_{\mathfrak{t}_\mu, w-2k }(z)E_{\mathfrak{t}_\mu, 2}^k(z)
    \end{eqnarray*}
for some $ B_{\mathfrak{t}_\mu,w - 2k} \in \mathcal{M}_{w-2k}(\mathfrak{H}(\varpi_\mu)) $. So, given elements of the vector space $ \mathcal{QM}^{r}_{w}(\mathfrak{H}(\varpi_\mu)) $, these functions are homogeneous with respect to weight when expanded in a sum. When $w = 4k $, the dimension of $ \mathcal{QM}^{r}_{w}(\mathfrak{H}(\varpi_\mu) $ can be determined in conjunction with Theorem \ref{thm:HeckeModularDimension} by computing 
    \begin{eqnarray*}
        \sum_{k=0}^{r}\dim \mathcal{M}_{w-2k}(\mathfrak{H}(\varpi_\mu)).
    \end{eqnarray*}
Dimension formulas can be given completely for special cases $ \mu = 3 $, $ 4 $, $ 6 $.

   \begin{exa}\label{dijgraaf}
        We want to provide an example of a quasiautomorphic form from a vector space of depth larger than $ r = 1 $; these objects have begun to appear in more and more places. For example, it was shown in \cite{dijkgraaf1995} that the quasimodular form
            \begin{eqnarray*}
                \sum_{n=0}^{3} B_{\mathfrak{t}_3, 6-2n}E_{\mathfrak{t}_3, 2}^n &=&
                \frac{5 E_{\mathfrak{t}_3, 2}^3 -3E_{\mathfrak{t}_3,4}E_{\mathfrak{t}_3, 2} - 2E_{\mathfrak{t}_3, 6}}{51840}  \qquad \qquad
                \\ &=& q^2 + 8q^3 + 30q^4 + 80q^5 + 180 q^6+ 336 q^7 +  \cdots 
            \end{eqnarray*}
        is a generating function for some objects from string theory (generating function counts the number of genus $ g = 2 $  covers of a $ g = 1 $ elliptic curve of degree $ n $). The up-shot is the sequence of integer coefficients 
            \begin{eqnarray*}
                0 , 0, 1, 8, 30, 80, 180, 336, ... 
            \end{eqnarray*}    
        has an interesting combinatorial interpretation\footnote{See A126858 at www.oeis.org}. See also \cite{mazur2004} for further discussion of this and some related functions. Dijkgraaf's result is generalized to genus $ 6g-6 $ in \cite{kanekoZagier1995}.
    \end{exa}

    \begin{defin}[Ramanujan-Serre operator]\label{Ramanujan-Serre}
        Let $ f(z) \in \mathcal{M}_w(\mathfrak{H}(\varpi_\mu))$ have an expansion in terms of $ {q := e^{2 \pi i z/\varpi_\mu}} $ or $ g(q)$ is a Fourier series. We write $ \theta f(z) = q \frac{d \,g(q)}{dq} $. Then the so-called \emph{Ramanujan-Serre differential operator} $ \partial_{\mathfrak{t}_\mu, w} $ is defined by
            \begin{eqnarray}\label{generalizedSerre}
                \partial_{\mathfrak{t}_\mu, w} f &:=& \theta f  - \frac{w}{2}\left(1 - \frac{1}{2} -\frac{1}{\mu}\right) E_{\mathfrak{t}_\mu, 2} f\\ &=& \theta f  - \frac{w(\mu-2)}{4\mu} E_{\mathfrak{t}_\mu, 2} f
            \end{eqnarray}
        (recall that necessarily here $ w = 2k $ holds for some $ 2 \leq k $).
    \end{defin}

As is customary, we define the iterated Ramanujan-Serre derivative inductively in
    \begin{notn}
        If $ k \in \mathbb{N}\cup \{0\} $, then for the iterated Ramanujan-Serre differential operator we write
        \begin{eqnarray*}
            \partial_{\mathfrak{t}_\mu, w}^0f &:=& f,\\ 
            \partial^{k+1}_{\mathfrak{t}_\mu, w}f &:=& \partial_{\mathfrak{t}_\mu, w+2k}(\partial^{k}_{\mathfrak{t}_\mu, w} f).
        \end{eqnarray*}
    \end{notn}

An important fact about $ \partial_{\mathfrak{t}_\mu, w }  $ is that each application increases the weight of the function by $ 2 $. In other words, suppose that $ f $ is an quasiautomorphic form of weight $ w $. Then $ \partial_{\mathfrak{t}_\mu, w }f $ has weight $ w+2 $. Also, it is not hard to show that $ \partial_{w} $ obeys the product rule or
    \begin{eqnarray*}
        \partial_{\mathfrak{t}_\mu, w_f + w_g}(fg) = \left(\partial_{\mathfrak{t}_\mu, w_f}f\right)g  + f\left(\partial_{\mathfrak{t}_\mu, w_g}g\right) 
    \end{eqnarray*}
holds where notation here is obvious.

\section{Introducing Hecke vector-forms}
Having introduced the background that we need, we now introduce the Hecke vector-forms. This is a method of representing a quasiautomorphic form $U_{\mathfrak{t}_\mu, w, r} $ as a vector function, whereas this vector function has a simpler transformation behavior under the associated generators $T_{\mathfrak{t}_\mu}z$ and $S_{\mathfrak{t}_\mu}z $.

\subsection{The idea of a vector-form}
Because this section is quite long and the build-up to the Fundamental Theorem is technical, we give a condensed explanatory summary beforehand so that the reader knows our aim.

\begin{enumerate}
    \item Given some quasiautomorphic form $ U_{\mathfrak{t}_\mu, w, r} $, then, as the depth $ r $ increases, the function becomes increasingly complicated under elements of $ \mathfrak{H}(\varpi_\mu) $. Our solution is to encode $U_{\mathfrak{t}_\mu, w, r} $ in a vector function that will then have a simpler transformation behavior. These vector functions -- what we call the Hecke vector-forms -- are vector analogues of automorphic forms. To wit, we first introduce vector components $ g_{U,\ell} $ (see Definition \ref{def:components} for details) that depend on the function $ U_{\mathfrak{t}_\mu, w, r} $, and it is this $ g_{U,\ell} $ for $0 \leq \ell \leq r $ in the final construction that has many special properties. Much of the technicality is in proving what we need for these component functions, though all proofs are strictly elementary.
    \begin{figure}[H]
        \centering
            \begin{tikzcd}
                & U_{\mathfrak{t}_\mu, w, r}(z) \ar[dl] \ar[d] \ar[dr] & \\
                g_{U,0}(z) & g_{U,1}(z) \quad \cdots &  g_{U,r}(z)        
            \end{tikzcd}
        \end{figure}
    \item After an introduction of $ g_{U,\ell}(z) $, we introduce binomial analogues in the functions \[ { f_{U,k}(z) =  z^k \sum_{\ell=0}^{\min(k,r)}\binom{k}{\ell}  \frac{g_{U, \ell}(z)}{z^\ell}} \] for $ 0 \leq k \leq r $ (again, see Definition \ref{def:components} for details). We must establish some properties of these functions as well, with proofs very much the same as those for determining properties of $ g_{U,\ell}$. A Hecke vector-form $\vec{F}_U$ with respect to $ U_{\mathfrak{t}_\mu, w, r} $ (in the simplest case) then, is just
        \begin{eqnarray*}
            \vec{F}_U(z) =
                \begin{pmatrix}
                    f_{U,0}(z) \\
                    f_{U,1}(z) \\
                    \vdots \\
                    f_{U,r}(z)
                \end{pmatrix}.
        \end{eqnarray*}
    It is the transformation properties of $ \vec{F}_U $ that are given in the Fundamental Theorem and are the primary motivation of our work.
    
    \item In the last stage, we prove the Fundamental Theorem of Hecke vector-forms; this gives the transformation behavior under the generators of the triangle group $ \mathfrak{H}(\varpi_\mu) $. We will quote the fruit of the labor now, leaving out some minutiae for the appropriate sections. Where $ p_U^\vee $ and $\mathbf{a}_U $ can be thought of as natural multipliers, we show that
        \begin{eqnarray*}
            \vec{F}_{{U}}\left(T_{\mathfrak{t}_\mu}z, {\varpi}^{-1}\right) &=& p_U^{\vee}\vec{F}_U(z), \\
            \frac{\vec{F}_{{U}}(S_{\mathfrak{t}_\mu}z)}{z^{w-r}} &=&  \pm\mathbf{a}_U \vec{F}_{U}(z)
        \end{eqnarray*}        
    hold. This behavior of Hecke vector-forms should be compared to the transformation behavior of automorphic forms under generators $ T_{\mathfrak{t}_\mu}z $ and $ S_{\mathfrak{t}_\mu}z $, which we reproduce for convenience: suppose that $ F_w $ is an automorphic form of weight $ w $ on $\mathfrak{H}(\varpi_{\mu}) $. Then 
        \begin{eqnarray*}
            F_w(T_{\mathfrak{t}_\mu}z) &=& F_w(z), \\
            \frac{F_w(S_{\mathfrak{t}_\mu}z)}{z^w} &=& F_w(z)
        \end{eqnarray*}
    hold. 
\end{enumerate}

\subsection{Required matrix theory}\label{matrixTheory}

For completeness we give a very concise summary of basic matrix notions; not everything is strictly review, however, as we also give a few idiosyncratic notions that make our task more convenient. For reference, an encyclopedic source of matrix theory is \cite{hornJohnson1985}. 

An $ m \times n $ \emph{matrix} $ A $ over a field $ \mathbb{F} $ we mean the array
    \begin{equation*}
        A :=
            \begin{pmatrix}
                a_{0,0} & a_{0,1} & \cdots & a_{0,n} \\
                a_{1,0} & a_{2,2} & \cdots & a_{2,n} \\
                \vdots &\vdots & \ddots & \vdots \\
                a_{m,0} & a_{m,1} & \cdots & a_{m,n} 
            \end{pmatrix}    
    \end{equation*}
consisting of $ m+1 $ rows and $ n+1 $ columns where $ a_{i,j} \in \mathbb{F} $.  Another more concise way to notate a matrix is by the double index method i.e. $ A = (a_{i,j}) $, which we shall frequently use.

For our work, we will be interested only in the special case of square $ (r+1) \times (r+1) $ matrices and $ (r+1) \times 1 $ row or $ 1 \times (r+1) $ column matrices. We treat vectors  $ x = (x_0, x_1, \cdots, x_r) $ as column or row matrices as made clear by the context. Because the matrices we use are either square or simply vectors, we do not have to introduce some standard concepts like e.g. \emph{conformable} matrices. 

\begin{warn}
    Throughout we will index entries, say $ a_i $, of our matrices, beginning from $ i = 0 $. Thus, a square matrix $ A_r $ will have order $ (r+1) \times (r+1) $. This practice of indexing shall greatly improve readability.   
\end{warn}

We briefly review some typical operations with matrices. We denote arbitrary $ (r+1) \times (r+1) $ matrices by $ A_r, B_r, C_r $, etc. The elements from the ambient field or function space we denote by $ a_{i,j}, b_{k,l} $, etc.

The following are all standard matrix operations:
        \begin{enumerate}
            \item By \emph{scalar multiplication} we mean $ \alpha A = \alpha ( a_{i,j}) := ( \alpha a_{i,j} ) $ where $ \alpha \in \mathbb{C} $.

            \item By \emph{matrix addition} we mean $ A+B = ((a_{i,j}) + (b_{i,j}) ) := ( a_{i,j} + b_{i,j} ) $.
        
            \item By \emph{matrix multiplication} we mean $ AB = (a_{i,j})(b_{i,j}) := (  \sum_{k=1}^{r}a_{ik}b_{kj} ) $. 
        \end{enumerate}

Next, we introduce some notation for a selection of important special matrices. 
        \begin{defin}
            The \emph{exchange matrix} is defined by
            \begin{eqnarray}\label{exchangeMatrix}
                \iota_r  :=
                    \begin{pmatrix}
                        &  &  & 1^0 \\
                        &  & 1^1 &  \\
                        &  \iddots &  &  \\
                        1^r &  &  &  \\
                    \end{pmatrix}.
            \end{eqnarray} 
        \end{defin}
This exchange matrix (as it is called in \cite{hornJohnson1985}) is perhaps not typical, but it is important for us.
    \begin{defin}\label{def:vandermondeMatrix}
        The so-called \emph{Vandermonde matrix} is defined by 
            \begin{eqnarray}\label{vandermondeMatrix}
                \mathcal{V}_r(z_0, \cdots , z_{r}) :=
                    \begin{pmatrix}
                        1 & z_0 & \cdots & z_0^r \\
                        1 & z_1 & \cdots & z_1^r \\
                        \vdots & \vdots & \ddots & \vdots \\
                        1 & z_r & \cdots & z_r^r 
                    \end{pmatrix}  .
            \end{eqnarray}
    \end{defin}
The following two matrices are vital to our formulations and go hand-in-hand with the Vandermonde matrix. Thus, we introduce special notation:
    \begin{notn} 
        We denote
            \begin{eqnarray}\label{basisVector}
                B_r(z) := 
                    \begin{pmatrix}
                        1 &
                        z &
                        \cdots &
                        z^r
                    \end{pmatrix}^{\text{\emph{T}}}
            \end{eqnarray}
        and similarly  
            \begin{eqnarray}\label{basisMatrix}
                V_r(z) :=
                    \begin{pmatrix}
                        1 & &
                        \\
                        z & 1 &
                        \\
                        \vdots & \vdots & \ddots &
                        \\
                        z^r & z^{r-1} & \cdots & 1
                \end{pmatrix}.
            \end{eqnarray}
        \end{notn}

    \begin{warn}
        Throughout this exposition, we take predictable liberties with the notation. For example, we might write $ A_r(z) = A $, dropping the subscript $ r $ and the argument. When from the context it is clear, we will try to streamline notation.  
    \end{warn}

The following matrix operation is standard 
    \begin{enumerate}
        \item If $ A = ( a_{i,j} ) $, then the so-called \emph{transpose} of a matrix $ A $ is defined by $ A^{\text{T}} := ( a_{j,i} ) $.    
    \end{enumerate}
Matrix transpose is nothing other than ``reflecting" the matrix about the leading diagonal. We introduce two non-standard operations similar to matrix transpose. We have a great use of these operations.
    \begin{defin}\label{def:exchanges}
        We define the $x$-\emph{exchange} and the $y$-\emph{exchange} of matrix $ A_r $ by
            \begin{eqnarray*}
                A_r^{\emph{X}} &: =& \iota_r A_r,\\   
                A_r^{\emph{Y}} &: =&  A_r \iota_r,
            \end{eqnarray*}    
        respectively, where $\iota_r $ is the exchange matrix of equation \eqref{exchangeMatrix}.
    \end{defin}
The naming $ x $-exchange and $ y $-exchange refer to the axis about which the matrix is being ``reflected" by the operation. The matrices $ A^{\text{X}} $ and $A^{\text{Y}}$ occur completely naturally in our study alongside $ A^{\text{T}} $.

\begin{exa}
        To illustrate the behavior of our exchange operations, observe 
            \begin{eqnarray*}
                    \begin{pmatrix}
                        a &&& \\
                        b & c&& \\
                        \vdots & \vdots & \ddots & \\
                        d & e & \cdots & f
                    \end{pmatrix}^{\emph{\text{X}}} 
                =
                    \begin{pmatrix}
                        d & e & \cdots & f \\
                        \vdots & \vdots & \iddots & \\
                        b & c & \\
                        a & & 
                    \end{pmatrix}
            \end{eqnarray*}
        holds, etc.
    \end{exa}

 \begin{defin}[Triangular matrix]
        Let $ t_r^\vee $ be an $ (r+1) \times (r+1) $ matrix of the form
            \begin{eqnarray*}
                t_r^\vee = 
                    \begin{pmatrix}
                        a_{0,0} &&& \\
                        a_{1,0} & a_{1,1} && \\
                        \vdots & \vdots & \ddots & \\
                        a_{r,0} & a_{r,1} & \cdots & a_{r,r} \\
                    \end{pmatrix},
            \end{eqnarray*}
        where $a_{i,j}$ are from some field, say $\mathbb{C} $. Then $ t_r^{\vee} $ is a so-called \emph{lower triangular matrix}. If $ t_r^\vee $ is a lower triangular matrix, then $ t_r^{\vee \emph{\text{T}}} $ is called an \emph{upper triangular matrix}.  
    \end{defin}

    \begin{notn}\label{polyVectors}
        Throughout our essay, we shall denote special lower and upper triangular matrices with $ {}^{\wedge} $  and $ {}^{\vee } $, respectively. As an example, we shall henceforth write
            \begin{eqnarray*}
                V_r^{\vee}(z) =
                    \begin{pmatrix}
                        1 & &
                        \\
                        z & 1 &
                        \\
                        \vdots & \vdots & \ddots &
                        \\
                        z^r & z^{r-1} & \cdots & 1,
                \end{pmatrix}
            \end{eqnarray*}
        etc. With this notation, also observe that the identity $ I_r $ can be written as either an upper triangular or lower triangular matrix i.e. the relation
            \begin{eqnarray*}
                V_r^\vee(0) = I_r^{\vee} = I_r^{\wedge} = V_r^\wedge(0)
            \end{eqnarray*}
        holds.
    \end{notn}

    \begin{defin}[Toeplitz matrix, etc.]\label{def:toeplitz}
        Suppose that $ A_r $ is a matrix of the form
            \begin{eqnarray*}
                A_r = 
                    \begin{pmatrix}
                        a_0 & a_1 & \cdots & a_r \\
                        a_{-1} & a_0 & \ddots & \vdots \\
                         \vdots & \ddots & \ddots & a_1 \\
                        a_{-r} & \cdots & a_{-1} & a_0 \\
                    \end{pmatrix}.
            \end{eqnarray*}
        Then $ A_r $ is a so-called \emph{Toeplitz matrix} (constant in the diagonals). If $ A_r $ is a Toeplitz matrix then 
            \begin{eqnarray*}
                A_r^{\text{\emph{Y}}} = 
                    \begin{pmatrix}
                         a_r & \cdots & a_{1} & a_0\\
                         \vdots & \iddots & a_0 & a_{-1} \\
                         a_1 & \iddots & \iddots & \vdots  \\
                        a_0 & a_{-1} & \cdots & a_{-r} \\
                    \end{pmatrix}
            \end{eqnarray*}
        is a so-called \emph{Hankel matrix} (constant in the antidiagonals). 
    \end{defin}
Triangular matrices have many advantageous properties, so it benefits us to use these matrices over general square matrices in places where we have a choice. Toeplitz and Hankel matrices simplify in the sense that only a single index is needed to define them (see Definition \ref{def:toeplitz}). Thus, triangular Topelitz or triangular Hankel are the most preferred.    

There is another classical operation we set apart in
    \begin{defin}[Schur-Hadamard product]\label{SchurHadamardProd}
        Let $ A $ and $ B $ both be $ (r+1) \times (r+1) $ matrices. The \emph{Schur-Hadamard product} of $ A =(a_{i,j}) $ and $B =(b_{i,j}) $ is defined as
            \begin{eqnarray*}
                A \odot B := (a_{i,j}b_{i,j}) =
                    \begin{pmatrix}
                        a_{0,0}b_{0,0} & a_{0,1}b_{0,1} & \cdots & a_{0,r}b_{0,r} \\
                        a_{1,0}b_{1,0} & a_{1,1}b_{1,1} & \cdots & a_{1,r}b_{1,r} \\
                        \vdots & \vdots & \ddots & \vdots \\
                        a_{r,0}b_{r,0} & a_{r,1}b_{r,1} & \cdots & a_{r,r}b_{r,r}
                    \end{pmatrix}.
            \end{eqnarray*}
    \end{defin}

    \begin{notn}\label{notn:altExchange}
        We introduce shortened notation for the alternating exchange matrix. We write
            \begin{eqnarray*}
                \mathbf{a}_r := \mathcal{V}_r(-1,-1,\cdots,-1)\odot \iota_r. 
            \end{eqnarray*}
    \end{notn}

    \begin{rmk}
        Matrix product and Schur-Hadamard product will differ significantly in their arithmetic. For example, both have their respective identities that are not equivalent. One must also mind the order of operations. In this essay, we restrict to very basic manipulations with these matrices, and so we avoid an extensive reproduction of their algebraic properties.
    \end{rmk}

The following special matrices are often studied for their own sake (see e.g. \cite{edelmanStrang2004}, \cite{yamaleev2015}, \cite{yangMicek2007}, etc.). We will make great use of them and their remarkable character should be evident, despite being only a utility for us.
    \begin{defin}[Lower triangular Pascal matrix, etc.]\label{def:triPascalMatrix}
        The \emph{lower triangular Pascal matrix} is defined by 
            \begin{eqnarray*}
                p_{r}^{\vee} := 
                    \begin{pmatrix}
                        {\binom{0}{0}} & & &  \\
                        {\binom{1}{1}} & {\binom{1}{0}} & & \\
                        \vdots & \vdots & \ddots & \\
                        {\binom{r}{r}} & {\binom{r}{r-1}} & \cdots & {\binom{r}{0}}
                    \end{pmatrix}.
            \end{eqnarray*}
    Similarly, the \emph{upper triangular Pascal matrix} is defined by
            \begin{eqnarray*}
                p_r^{\wedge} := (p_{r}^{\vee})^{\emph{\text{T}}} =
                    \begin{pmatrix}
                        \binom{0}{0} & \binom{1}{1} & \cdots & 
                        \binom{r}{r}   \\
                        & {\binom{1}{0}} & \cdots &  {\binom{r}{r-1}} \\
                        &  & \ddots & \vdots \\
                        & & & {\binom{r}{0}}
                    \end{pmatrix}.
            \end{eqnarray*} 
    These matrices are clearly invertible.
    \end{defin}

Related to the above pascal matrices is a scaling matrix that we will need.     
    \begin{notn}
        We define the scaling binomial vectors
            \begin{eqnarray*}
                \mathbf{b}_r := 
                    \begin{pmatrix}
                        \binom{r}{0} &
                        \binom{r}{1} &
                        \cdots &
                        \binom{r}{r}
                    \end{pmatrix}^\text{\emph{T}}
            \end{eqnarray*}
        and
            \begin{eqnarray*}
                \mathbf{b}_r^{-1} := 
                    \begin{pmatrix}
                        \binom{r}{0}^{-1} &
                        \binom{r}{1}^{-1} &
                        \cdots &
                        \binom{r}{r}^{-1}
                    \end{pmatrix}^\text{\emph{T}}.
            \end{eqnarray*}
        Clearly, $ \mathbf{b}_r \odot \mathbf{b}^{-1}_r = (1, 1, \cdots, 1)^{\text{\emph{T}}} $ is true, etc.
    \end{notn}

The next Proposition allows us to combine various matrix ideas from this section into a single elegant identity. Furthermore, it serves as motivation and also as a model for how we develop the Hecke vector-forms later. Clearly, the binomial theorem is a deeply useful result from classical mathematics and we are eager to utilize it also here, in the context of automorphic forms, etc.  

\begin{prop}[Mixed-product binomial theorem]\label{genBinomialTheorem}
        Let $ B_r(z) $ and $ V^{\vee}_r(z) $ be as defined in equations \eqref{basisVector} and \eqref{basisMatrix}. 
        Then the relation 
            \begin{eqnarray*}
                  B_r(x+zy) = \left( p_r^{\vee} \odot V^\vee_r(x) \right) \left( B_r(y) \odot B_r(z) \right)  
            \end{eqnarray*}
        holds. 
    \end{prop}
    \begin{proof}
        This is a simple calculation. We see
            \begin{eqnarray*}
                \left( p_r^{\vee} \odot V^\vee_r(x) \right)  B_r(y)  &=&
                    \begin{pmatrix}
                        \binom{0}{0}x^0 &&&\\
                        \binom{1}{0}x^1 & \binom{1}{1}x^0 & & \\
                        \vdots & \vdots & \ddots & \\
                        \binom{r}{0}x^r & \binom{r}{1}x^{r-1} & \cdots & \binom{r}{r}x^0  
                    \end{pmatrix}
                    \begin{pmatrix}
                        y^0  \\
                        y^1  \\
                        \vdots \\
                        y^r 
                    \end{pmatrix} \\
                    &=&
                    \begin{pmatrix}
                        1 \\
                        (x+y) \\
                        \vdots \\
                        (x+y)^r
                    \end{pmatrix}\\
                    &=&  B_r(x+y)
            \end{eqnarray*}
        holds. Now under the substitution $y \mapsto yz $ by which $B_r(yz) = B_r(y) \odot B_r(z)$ holds, we are finished.
    \end{proof}
Our Fundamental Theorem is in a sense an analogue of Proposition \ref{genBinomialTheorem} involving quasiautomorphic forms. It is instructive to keep this in mind.

This concludes our introduction of basic matrix theory.

\subsection{Component functions $ g_{U,\ell} $ and $ f_{U,k} $}

In the following we will use the notation $ \mathfrak{H}(\varpi_\mu) $ or $ \mathfrak{t}_\mu $ according to readability. Starting out, we fix $ {U_{\mathfrak{t}_\mu, w,r}(z) \in \mathcal{QM}^r_w(\mathfrak{H}(\varpi_\mu)) } $ as some arbitrary quasiautomorphic form of weight $ w $ and depth $ r $ on the Hecke triangle group $ \mathfrak{t}_\mu = \mathfrak{H}(\varpi_\mu) $. We know that for $z\in \mathbb{H}$, $ U_{\mathfrak{t}_\mu, w, r}  $ can be written as a sum of the form
    \begin{equation}\label{quasiautoSumForm}
        U_{\mathfrak{t}_\mu, w,r}(z) = \sum_{m=0}^{r}h_{U, m}(z) \,E_{\mathfrak{t}_\mu, 2}^m(z),
    \end{equation}
where $ h_{U, m} $ is an automorphic form of weight $ w - 2m $  on $ \mathfrak{H}(\varpi_\mu) $ and $ E_{\mathfrak{t}_\mu, 2} $ is the unique quasiautomorphic form of weight $ 2 $ on $ \mathfrak{H}(\varpi_\mu) $ (e.g. see \cite{grabner2020} for the case $ \mu = 3 $). 

    \begin{warn}
        To ease a use of notation, notice in \eqref{quasiautoSumForm} that $ h_{U, m} $ are not automorphic forms of weight $ m $, but of weight $ w-2m $. We have restricted this kind of notation to the letter $ h $ to avoid confusion with the convention of indicating the explicit weight of an automorphic form or quasiautomorphic form as a subscript. 
    \end{warn}

It will motivate later constructions to see how $ U_{\mathfrak{t}_\mu, w, r}(z) $ behaves under the generator $ {S_{\mathfrak{t}_\mu}z = -1/z} $. Moreover, it is clear by observation of equation \eqref{quasiautoSumForm} that $ U_{\mathfrak{t}_\mu, w, r}(z) $ is invariant under $ T_{\mathfrak{t}_\mu}z = z+\varpi $ by periodicity.

    \begin{prop}[$U$ under $Sz$]\label{quasiUnderS}
        Recall that the structure constant $ C_{\mathfrak{t}_\mu} $ was defined at Definition \ref{constantOfVanishing}. Let $ U_{\mathfrak{t}_\mu, w, r} $ be arbitrary and if $ S_{\mathfrak{t}_\mu }z = -1/z $ is a generator, then
            \begin{eqnarray*}
                \frac{U_{\mathfrak{t}_\mu, w, r}(S_{\mathfrak{t}_\mu}z)}{z^{w-r}} =  z^r\sum_{\ell=0}^{r}\binom{r}{\ell}\frac{g_{U, \ell}(z)}{z^\ell}
            \end{eqnarray*}
        holds, where 
            \begin{equation*}
                \binom{r}{\ell}  g_{U, \ell}(z) := C_{\mathfrak{t}_\mu}^\ell \sum_{m=0}^{r-\ell} \binom{\ell + m}{m}h_{U, \ell + m}(z)E_{\mathfrak{t}_\mu, 2}^m(z).
            \end{equation*}
    \end{prop}
    \begin{proof}
        Firstly, from the sum expansion of $U_{\mathfrak{t}_\mu, w, r} $ above we find
            \begin{eqnarray*}
                z^rU_{\mathfrak{t}_\mu, w, r}(S_{\mathfrak{t}_\mu}z) &=& z^r \sum^{r}_{\ell=0}h_{U, m}(S_{\mathfrak{t}_\mu}z)E^m_{\mathfrak{t}_\mu,2}(S_{\mathfrak{t}_\mu}z)\\
                &=& z^{w+r} \sum^{r}_{m=0}h_{U, m}(z)\left(E_{\mathfrak{t}_\mu,2}+ \frac{C_{\mathfrak{t}_\mu}}{z}\right)^m \\
                &=& z^{w+r} \sum^{r}_{m=0}h_{U, m}(z)\left(\sum^{m}_{\ell=0} \binom{\ell}{m}E_{\mathfrak{t}_\mu, 2
                }^{m-\ell} \left(\frac{C_{\mathfrak{t}_\mu}}{z}\right)^{\ell}\right)
            \end{eqnarray*}
        all hold. Now, applying interchange laws for sums, we know
            \begin{eqnarray*}
                &{}& z^{w+r} \sum^{r}_{m=0}h_{U, m}(z)\left(\sum^{m}_{\ell=0} \binom{m}{\ell}E_{\mathfrak{t}_\mu, 2
                }^{m-\ell} \left(\frac{C_{\mathfrak{t}_\mu}}{z}\right)^{\ell}\right)\\ &=& 
                z^{w+r} \sum^{r}_{m=0}\sum^{m}_{\ell=0}  \binom{m}{\ell}h_{U, m}(z) E_{\mathfrak{t}_\mu, 2
                }^{m-\ell} \left(\frac{C_{\mathfrak{t}_\mu}}{z}\right)^{\ell} 
            \end{eqnarray*}
        holds. Utilizing in particular the summation law
            \begin{eqnarray*}
                \sum_{m=0}^{r} \sum^{m}_{\ell=0} a_{\ell, m} &=& \sum_{\ell=0}^{r} \sum^{r-\ell}_{m=0} a_{\ell, \ell + m},
            \end{eqnarray*}
        we find
            \begin{eqnarray*}
                &{}& z^{w+r} \sum^{r}_{m=0}\sum^{m}_{\ell=0}  \binom{m}{\ell}h_{U, m}(z) E_{\mathfrak{t}_\mu, 2
                }^{m - \ell} \left(\frac{C_{\mathfrak{t}_\mu}}{z}\right)^{\ell} \\ 
                &=& z^{w+r} \sum^{r}_{\ell=0}\sum^{r-\ell}_{m=0}  \binom{\ell + m}{\ell}h_{U, \ell + m}(z) E_{\mathfrak{t}_\mu, 2
                }^m(z) \left(\frac{C_{\mathfrak{t}_\mu}}{z}\right)^{\ell} \\
                &=& z^{w+r} \sum^{r}_{\ell=0} \left(\frac{C_{\mathfrak{t}_\mu}}{z}\right)^{\ell} \sum^{r-\ell}_{m=0}  \binom{\ell + m}{m}h_{U, \ell + m}(z) E_{\mathfrak{t}_\mu, 2
                }^m(z)  \\
                &=& z^{w}\left( z^r\sum^{r}_{\ell=0} \binom{r}{\ell}\frac{g_{U, \ell}(z)}{z^{\ell}} \right)
            \end{eqnarray*}
        holds, where in the penultimate line we have applied the law
            \begin{eqnarray*}
                \binom{\ell + m}{\ell} = \binom{\ell + m }{m}.
            \end{eqnarray*}
        In summary, we have shown
            \begin{eqnarray*}
                z^rU_{\mathfrak{t}_\mu, w, r}(S_{\mathfrak{t}_\mu}z) = z^{w}\left( z^r\sum^{r}_{\ell=0} \binom{r}{\ell}\frac{g_{U, \ell}(z)}{z^{\ell}} \right)
            \end{eqnarray*}
        holds or what we wanted. This completes the proof.
    \end{proof}

    \begin{rmk}\label{rem:binomialForm}
        To not get lost in superficial technicality, we note that we are only ever dealing with binomial structure. That is, for a formal variable $\mathfrak{a}$ , we have
            \begin{eqnarray*}
                \frac{U_{\mathfrak{t}_\mu, w, r}(S_{\mathfrak{t}_\mu}z)}{z^{w-r}} 
                &=&  \sum^{r}_{\ell=0} \binom{r}{\ell}\mathfrak{a}^\ell z^{r-\ell}\\
                &=&  (\mathfrak{a} + z)^{r} 
            \end{eqnarray*}
        holds where $\mathfrak{a}^\ell\mapsto g_{U, \ell }(z) $. Similarly, observe that 
            \begin{eqnarray}
                \sum_{k=0}^{n}\binom{n}{k}g_{U, k}z^{n-k} = \sum_{k=0}^{n}\binom{n}{k}g_{U, n-k}z^{k} 
            \end{eqnarray}
        is true, completely in accord with the familiar relation for binomial sums.
    \end{rmk}
    
The proof of Proposition \ref{quasiUnderS} was strictly elementary, but it serves as a simple template for later arguments. We also have motivated the introduction of $ g_{U, \ell} $, a function having a significant role in what is to come.

We now fix notation $g_{U, \ell}$ from Proposition  \ref{quasiUnderS} and introduce another useful function $ f_{U, k } $ as was done in \cite{grabner2020}. The functions are perhaps technical on the surface, but their convenience becomes clear later when they are components of matrices. In matrices, their transformation properties are easy to summarize in basic formulas.

    \begin{defin}[Components $ g_{U, \ell} $ and $ f_{U, k}$]\label{def:components}
        Suppose that $ z \in \mathbb{H}$ and let $U_{\mathfrak{t}_\mu, w, r} $ be arbitrary. For the structure constant $ C_{\mathfrak{t}_\mu} $, see Definition \ref{constantOfVanishing}. Then we define the \emph{component functions} by
            \begin{eqnarray}\label{auxG}
                g_{U, \ell}(z) : =  C^{\ell}_{\mathfrak{t}_\mu}\binom{r}{\ell}^{-1} \sum_{m=0}^{r-\ell} \binom{\ell+m}{m} h_{U, \ell+m}(z)E_{\mathfrak{t}_\mu, 2}^m(z)
            \end{eqnarray}
        and 
            \begin{eqnarray}\label{auxF}
                f_{U, k}(z) := z^k \sum_{\ell=0}^{\min(k,r)}\binom{k}{\ell}  \frac{g_{U, \ell}(z)}{z^\ell}.
            \end{eqnarray}
    \end{defin}

   \begin{rmk}
        Observe from \eqref{auxG} that $ g_{U, \ell}(z) $ is a a weighted truncation of $ U_{\mathfrak{t}_\mu, w,r}(z) $ in \eqref{quasiautoSumForm}. Similarly, by \eqref{auxF} $f_{U, k}(z) $ is a truncation of $ U_{\mathfrak{t}_\mu, w,r}(S_{\mathfrak{t}_\mu}z) $ from Proposition \ref{quasiUnderS}. We could be more pedantic in Definition \ref{def:components} and write, say, $ g_{U, \ell}(z, C_{\mathfrak{t}_\mu}) $ to emphasize the dependency, but we avoid this.  
    \end{rmk}

We next establish an assortment of identities for our functions $ g_{U, \ell} $ and $f_{U, k} $, all toward the end of developing Hecke vector-forms.

    \begin{prop}($ g_U $ under $Sz$)\label{auxGUnderS}
        Let $U_{\mathfrak{t}_\mu, w, r} $ be arbitrary. The relation
            \begin{eqnarray*}
                 \frac{g_{U,\ell}(S_{\mathfrak{t}_\mu} z)}{z^{w-r-\ell}} =   z^{r-\ell}\sum_{m=0}^{r-\ell}\binom{r-\ell}{m} \frac{g_{U, \ell + m}(z)}{z^m}  
            \end{eqnarray*}
        holds. 
    \end{prop}
    \begin{proof}
        First, recall that $ S_{\mathfrak{t}_\mu}z := -1/z $. For the proof, we will manipulate both sides and meet in the middle. We find by using the transformation properties of $ h_{U, m} $ and $ E_{\mathfrak{t}_\mu, 2} $ that
            \begin{eqnarray*}
                &{}& z^{r+\ell}g_{U, \ell}(S_{\mathfrak{t}_\mu}z)\\ &=& z^{w+r-\ell} \, C_{\mathfrak{t}_\mu}^\ell \sum^{r-\ell}_{m=0}\frac{\binom{\ell + m}{m}}{\binom{r}{\ell}}h_{U,\ell +m}(z) \left(E_{\mathfrak{t}_\mu, 2}(z) + \frac{C_{\mathfrak{t}_\mu}}{z} \right)^m 
                \\
                &=&  z^{w+r-\ell} \, C_{\mathfrak{t}_\mu}^\ell \sum^{r-\ell}_{m=0} \sum_{k-0}^{m}\frac{\binom{\ell + m}{m} \binom{m}{k}}{\binom{r}{\ell}} h_{U,\ell +m}(z)E^{m-k}_{\mathfrak{t}_\mu, 2}(z)\left( \frac{C_{\mathfrak{t}_\mu}}{z}\right)^{k} 
                \\
                &=&z^{w} \left[ z^{r-\ell} \left( \sum^{r-\ell}_{m=0} C_{\mathfrak{t}_\mu}^\ell \sum_{k=0}^{m}\frac{\binom{\ell + m}{m}  \binom{m}{k}}{\binom{r}{\ell}} h_{U,\ell +m}(z)E^{m-k}_{\mathfrak{t}_\mu, 2}(z) \left( \frac{C_{\mathfrak{t}_\mu}}{z}\right)^{k} \right) \right]
            \end{eqnarray*}
        holds. Thus, we are left to show
            \begin{equation*}
                \sum^{r-\ell}_{m=0} C_{\mathfrak{t}_\mu}^\ell \sum_{k=0}^{m}\frac{\binom{\ell + m}{m}  \binom{m}{k}}{\binom{r}{\ell}} h_{U,\ell +m}(z)E^{m-k}_{\mathfrak{t}_\mu, 2}(z) \left( \frac{C_{\mathfrak{t}_\mu}}{z}\right)^{k} 
                = \sum_{p=0}^{r-\ell}\binom{r-\ell}{p} \frac{g_{U, \ell + p}(z)}{z^p}
            \end{equation*}
        holds. Firstly, using the sum law
            \begin{eqnarray*}
                \sum_{m=0}^{r-\ell} \sum^{m}_{k=0} a_{k, m} &=& \sum_{k=0}^{r-\ell} \sum^{r-\ell-k}_{m=0} a_{k, m+k},
            \end{eqnarray*}
        we have to show
            \begin{eqnarray*}
                &{}& \sum^{r-\ell}_{k=0} \frac{\left(C_{\mathfrak{t}_\mu}^{\ell + m} \sum_{m=0}^{r-(\ell + k)}\frac{\binom{\ell + m+k}{m+k}  \binom{m+k}{k}}{\binom{r}{\ell}} h_{U,\ell +m+k}(z)\,E^m_{\mathfrak{t}_\mu, 2}(z)\right)}{z^k} \\ 
                &=& \sum_{p=0}^{r-\ell}\binom{r-\ell}{p} \frac{g_{U, \ell + p}(z)}{z^p}
            \end{eqnarray*}
        holds. At the left side, inner sum,  observe that binomial identity 
            \begin{eqnarray*}
                \frac{\binom{\ell + m + k}{m+k} \binom{m+k}{k}}{\binom{r}{\ell}} = \frac{(\ell + m +k)! \,(r-\ell)!}{m!\,k!\,r!} =  \frac{\binom{r-\ell}{m}\binom{\ell + m + k}{k}}{\binom{r}{\ell + m}}
            \end{eqnarray*}
        holds. Therefore, changing what needs to be changed we have what we wanted. 
    \end{proof}    
    
    \begin{prop}[$f_U$ under $Sz$]\label{prop:auxFUnderS}
        Let $ U_{\mathfrak{t}_\mu, w, r} $ be arbitrary. The relation
            \begin{eqnarray*}
                \frac{f_{{U}, k}(S_{\mathfrak{t}_\mu} z)}{{z^{w-r}} } = (-1)^k   { f_{U, r-k}(z)} 
            \end{eqnarray*}
        holds.
    \end{prop}
    \begin{proof}
        We begin by rewriting
            \begin{eqnarray*}
                f_{U, k}(S_{\mathfrak{t}_\mu}z) = \sum^{k}_{\ell=0}\binom{k}{\ell}g_{U, \ell}(S_{\mathfrak{t}_\mu} z) (S_{\mathfrak{t}_\mu} z)^{k-\ell}
                = (S_{\mathfrak{t}_\mu}z)^k\sum^{k}_{\ell=0}(-1)^\ell\binom{k}{\ell}\frac{g_{U, \ell}(S_{\mathfrak{t}_\mu} z)}{(S_{\mathfrak{t}_\mu}z)^\ell}
            \end{eqnarray*}
        holds. Then by Proposition \ref{auxGUnderS}, we can substitute a sum for $g_{U, \ell}(S_{\mathfrak{t}_\mu} z) $, which renders
            \begin{eqnarray*}
                f_{U, k}(S_{\mathfrak{t}_\mu}z) &=& (-1)^k  z^{-k}\sum^{k}_{\ell=0}(-1)^\ell\binom{k}{\ell}  \left( z^{w-2\ell}\sum_{m=0}^{r-\ell}\binom{r-\ell}{m} \frac{g_{U, m+\ell}(z)}{z^{m-\ell}} \right) \\
                &=& (-1)^k z^{w-k}\sum^{k}_{\ell=0}(-1)^\ell\binom{k}{\ell}  \left(\sum_{m=0}^{r-\ell}\binom{r-\ell}{r-\ell - m} \frac{g_{U, m+\ell}(z)}{z^{m+\ell}} \right) \\
                &=& (-1)^k z^{w-k}\sum^{k}_{\ell=0}\sum_{m=0}^{r-\ell}(-1)^\ell\binom{k}{\ell}  \binom{r-\ell}{r -(\ell + m)} \frac{g_{U, m+\ell}(z)}{z^{m+\ell}} \\
                &=& (-1)^k z^{w-k}\sum_{m=0}^{r} \frac{g_{U, m}(z)}{z^{m}} \left( \sum^{k}_{\ell=0}(-1)^\ell\binom{k}{\ell}  \binom{r-\ell}{r - m} \right) \\
                &=& (-1)^k z^{w-r}\sum_{m=0}^{r} {g_{U, m}(z)}{z^{r-k -m}} \left( \sum^{k}_{\ell=0}(-1)^\ell \binom{k}{\ell}  \binom{r-\ell}{r - m} \right)
            \end{eqnarray*}
        true. In the previous string of deductions we have used laws of sums; we used a unit multiplier $ {z^r}/{z^r}$; and we used the elementary fact that the relation
            \begin{eqnarray*}
                \binom{n}{m} = \binom{n}{n - m} 
            \end{eqnarray*}
        holds. We must deal next with the following binomial convolution
            \begin{eqnarray*}
                \sum^{k}_{\ell=0}(-1)^\ell \binom{k}{\ell}  \binom{r-\ell}{r - m}
            \end{eqnarray*}
        in the final line above. From \cite{boyadzhiev2018}, p. 140, we find the identity
            \begin{eqnarray}\label{eqn:binomConvol}
                \sum_{\ell=0}^k (-1)^k\binom{k}{\ell}\binom{y-\ell}{p} = \binom{y-k}{y-p}.
            \end{eqnarray}
        Thus, letting $ y \mapsto r $ and $p \mapsto r- m $ in (\ref{eqn:binomConvol}), we transform our convolution to find that 
            \begin{eqnarray*}
                \sum^{k}_{\ell=0}(-1)^\ell \binom{k}{\ell}  \binom{r-\ell}{r - m} = \binom{r-k}{m}
            \end{eqnarray*}
        holds. In conclusion, after applying the $ \min $ condition in the definition of $f_{U , k} $, we may write
            \begin{eqnarray*}
                f_{{U}, k}(S_{\mathfrak{t}_\mu} z) = (-1)^k z^{w-r}\sum_{m=0}^{r-k} \binom{r-k}{m}{g_{U, m}(z)}{z^{r-k -m}} =
                (-1)^k z^{w-r}f_{U, r-k}(z)
            \end{eqnarray*}
        and this completes the proof.
    \end{proof}

    \begin{prop}[$f_U$ under $ T^az $]\label{prop:fWithWeights}
            Let $ U_{\mathfrak{t}_\mu, w, r} \in \mathcal{QM}^r_w(\mathfrak{H}(\varpi_\mu))$ be arbitrary, $  a \in \mathbb{N} $, and $ {T_{\mathfrak{t}_\mu} z =  z + \varpi} $. Here, $ T_{\mathfrak{t}_\mu}^az = T_{\mathfrak{t}_\mu}z \circ T_{\mathfrak{t}_\mu} z \circ \cdots \circ T_{\mathfrak{t}_\mu}z $, composing $ a $ times, etc. The relation
            \begin{eqnarray*}
                f_{{U}, k}(T^a_{\mathfrak{t}_\mu} z) 
                = 
                \sum_{p=0}^{k}\binom{k}{p}\left( z^p\sum_{m=0}^{p} \binom{p}{m} \frac{ (a\varpi)^m g_{U, m}(z)}{z^m} \right)
            \end{eqnarray*}
        holds.
    \end{prop}
     \begin{proof}
        By the invariance of $g_{U, \ell }(z) $ under the generator $ T_{\mathfrak{t}_\mu} z $ and elementary algebra, we find  
            \begin{eqnarray*}
                f_{U, k}(T_{\mathfrak{t}_\mu }^a z) 
                &=& 
                \sum_{\ell=0}^{k} \binom{k}{\ell}g_{U , k-\ell}(z)(z + a\varpi)^\ell \qquad \qquad \\
                &=& 
                \sum_{\ell=0}^{k} \binom{k}{\ell}g_{U ,k-\ell}(z) \left( \sum_{m=0}^{\ell}\binom{\ell}{m} \left(a\varpi\right)^{\ell -m}z^m\right) \\
                &=&
                \sum_{p=0}^{k} \binom{k}{p} \sum_{m=0}^{p}\binom{p}{m}\left(a\varpi\right)^{p -m} g_{U, p-m}(z) z^m
            \end{eqnarray*}
        holds. Now we use the numbered identity from Remark \ref{rem:binomialForm} to write
            \begin{eqnarray*}
                f_{U, k}(T_{\mathfrak{t}_\mu }^a z) &=&
                \sum_{p=0}^{k} \binom{k}{p} \sum_{m=0}^{p}\binom{p}{m}\left(a\varpi\right)^{m} g_{U, m}(z) z^{p-m} \\
                &=& 
                \sum_{p=0}^{k} \binom{k}{p} \left( z^p \sum_{m=0}^{p}\binom{p}{m}g_{U, m}(z)\left(\frac{a\varpi}{z}\right)^m\right)
            \end{eqnarray*}
        is true. 
    \end{proof}
   
    \begin{cor}($f_U$ under $Tz$)\label{cor:auxFUnderT}
        With the assumptions of the previous Proposition, the relation 
            \begin{eqnarray*}
                f_{{U}, k}(T_{\mathfrak{t}_\mu} z) 
                = 
                \sum_{p=0}^{k}\binom{k}{p}\left( z^p \sum_{m=0}^{p}  \binom{p}{m} \frac{\varpi^{m}\, g_{U , m}(z)}{z^m} \right)
            \end{eqnarray*}
        holds.
    \end{cor}
    \begin{proof}
        Let $ a = 1 $ in Proposition \ref{prop:fWithWeights}.
    \end{proof}

\subsection{Transformation formulas of a Hecke vector-form}
Now that we have defined the component functions and established some basic transformation properties of these functions, we can proceed to main object of this chapter. 
    \begin{notn}
        Recalling Definition \ref{deepVectorSpace}, we may always write
            \begin{equation*}
                U_{\mathfrak{t}_\mu, w,r}(z) = \sum_{m=0}^{r}h_{U, m}(z) E_{\mathfrak{t}_\mu, 2}^m(z),
            \end{equation*}
        if $ h_{U, m} $ is an automorphic form of weight $ w - 2m $  on $ \mathfrak{H}(\varpi_\mu) $ and $E_{\mathfrak{t}_\mu, 2} $ is the unique quasiautomorphic form of weight $ 2 $ on $ \mathfrak{H}(\varpi_\mu) $.
    \end{notn}
We define a special matrix function dependent on $ U_{\mathfrak{t}_\mu, w, r} $ in
\begin{defin}
        Let $ U_{\mathfrak{t}_\mu, w,r} $ be arbitrary and in the sum form immediately above. Then we define 
            \begin{eqnarray*}
                {H}_{U}^{\wedge}(z) := 
                    \begin{pmatrix}
                        h_{U, r}(z) & h_{U, r-1}(z)  &\cdots  & h_{U, 0}(z) 
                        \\
                          & \ddots & \ddots   & \vdots 
                        \\
                        & & h_{U, r}(z)  & h_{U, r-1}(z) 
                        \\
                        & & & h_{U, r}(z)  
                    \end{pmatrix}.
            \end{eqnarray*}
    \end{defin}

The matrix $ H^\wedge_U $ houses all the automorphic parts of $ U_{\mathfrak{t}_\mu, w, r} $ and is both upper triangular and Toeplitz. This makes $ H^\wedge_U(z) $ a particularly nice matrix to work with. Notice also that the matrix identity $ H^{\wedge}_{U}({T}_{\mathfrak{t}_\mu}z ) = H^{\wedge}_{U}(z) $ can be read off. Under $ S_{\mathfrak{t}_\mu}z $ the matrix $ H^{\wedge}_{U} $ is only slightly more complicated.

    \begin{notn}
        Suppose that $ U_{\mathfrak{t}_\mu, w, r} $ is arbitrary. Then we write
            \begin{eqnarray*}
                 E_{U}(z)  := B_U(E_{\mathfrak{t}_\mu, 2}(z)) 
                 = 
                    \begin{pmatrix}
                        1 &
                        E^1_{\mathfrak{t}_\mu, 2}(z) &
                        \cdots &
                        E_{\mathfrak{t}_\mu, 2}^r(z)
                    \end{pmatrix}^{\text{\emph{T}}}.
            \end{eqnarray*}
        We might call $ E_{U} $ the \emph{quasiautomorphic basis} of $  U_{\mathfrak{t}_\mu, w, r}$.
    \end{notn}

    \begin{warn}
        To avoid confusion, we make the following clarification about notation. If we are given $ U_{\mathfrak{t}_\mu, w, r} $, then we can associate a matrix (or vector) of the form $ M_{U} $. Here $ M_U $ takes the depth of $ U_{\mathfrak{t}_\mu, w, r} $ meaning that $ M_U = M_r $ is a matrix (or vector) of order $ r+1 $.
    \end{warn}

    \begin{prop}\label{prop:gStack}
         See Definition \ref{def:components} for meaning of $ g_{U,\ell}(z)$. Suppose that $ U_{\mathfrak{t}_\mu, w, r} $ is arbitrary and $C_{\mathfrak{t}_\mu}$ is the structure constant given in Definition \ref{constantOfVanishing}. The relation
            \begin{eqnarray*}\label{hauptbuch}
                &{}&
                \begin{pmatrix}
                    g_{U, 0}(z) &
                    g_{U, 1}(z) &
                    \cdots
                    &
                    g_{U, r}(z)
                \end{pmatrix}^{\text{\emph{T}}}  
                \\ &=& \{B_U(C_{\mathfrak{t}_\mu})\odot \mathbf{b}^{-1}_U\} \odot \{\left( p^{\wedge}_U \odot H_U^\wedge(z) \right)^\text{\emph{Y}}E_{U}(z)\}
            \end{eqnarray*}
        holds. 
    \end{prop} 
    \begin{proof}
        We must mind the order of operations. Easily, we find that
            \begin{eqnarray*}
                \{B_U(C_{\mathfrak{t}_\mu})\odot \mathbf{b}^{-1}_U\}  = 
                    \begin{pmatrix}
                       \binom{r}{0}^{-1} & C_{\mathfrak{t}_\mu} \binom{r}{1}^{-1} & \cdots & C^r_{\mathfrak{t}_\mu} \binom{r}{r}^{-1} 
                    \end{pmatrix}^{\text{{T}}}
            \end{eqnarray*}
        holds. 
            
        Then, for the right most expression we have that 
            \begin{eqnarray*}
               &{}& ( p^{\wedge }_U \odot H_U^{\wedge} (z) )^\text{Y} {B}_{U}(E_2(z)) = ( p^{\wedge \text{Y} }_U \odot H_U^{\wedge\text{Y}} (z) ) E_{U}(z) \\ &=& \left\{
                    \begin{pmatrix}
                        {\binom{0}{0}} & \cdots & \binom{r-1}{r-1} & 
                        {\binom{r}{r}}   
                        \\
                        {\binom{1}{0}} &  \cdots &  {\binom{r}{r-1}} & 
                        \\
                        \vdots & \iddots  &  &  
                        \\
                        {\binom{r}{0}} & & & 
                    \end{pmatrix}
                \odot
                    \begin{pmatrix}
                        h_{U, 0}(z) & \cdots & h_{U, r-1}(z) & h_{U, r}(z) \\
                        h_{U, 1}(z) &  \cdots & h_{U, r}(z) \\
                        \vdots & \iddots & & \\
                        h_{U, r}(z)& & &  
                    \end{pmatrix}\right \}\\&\times&
                    \begin{pmatrix}
                        1   \\
                        E_{\mathfrak{t}_\mu,2}(z) \\
                        \vdots \\
                        E_{\mathfrak{t}_\mu, 2}^r(z) 
                    \end{pmatrix}\\
                    &=&
                    \begin{pmatrix}
                        \binom{0}{0} h_{U, 0} + \binom{1}{1} h_{U,1}(z)E_{\mathfrak{t}_\mu, 2}(z) + \cdots  +  \binom{r}{r} h_{U, r}(z)E_{\mathfrak{t}_\mu, 2}^r(z) 
                        \\
                        \binom{1}{0} h_{U, 1}(z) + \cdots + \binom{r}{r-1} h_{U, r}(z)E_{\mathfrak{t}_\mu, 2}^{r-1}(z) 
                        \\
                        \vdots 
                        \\
                        \binom{r}{0} h_{U, r}(z)
                    \end{pmatrix}
            \end{eqnarray*}
        holds. Combining both formulas using the Schur-Hadamard product completes the proof.    
    \end{proof}

By Proposition \ref{prop:gStack} we are compelled to isolate a
    \begin{defin}[Hauptbuch]\label{matrixDefVectForm}
        Let $ U_{\mathfrak{t}_\mu, w,r} $ be arbitrary. Then we call
            \begin{eqnarray*}
                G_U(z) := \begin{pmatrix}
                    g_{U, 0}(z) &
                    g_{U, 1}(z) &
                    \cdots
                    &
                    g_{U, r}(z)
                \end{pmatrix}^{\text{\emph{T}}}  
            \end{eqnarray*} 
        the \emph{hauptbuch} of $ U_{\mathfrak{t}_{\mu}, w,r} $.
    \end{defin}

This function $ G_U(z) $ contains vital information about the quasiautomorphic form $ U_{\mathfrak{t}_\mu, w, r}(z) $, hence the namesake (and German as a nod to Hecke). For example, in $ G_U(z) $ we find the structure constant, the automorphic substance, the vanishing orders, etc.

W have what we need to state the definition of a vector-form.   
    \begin{defin}[Hecke vector-form]\label{def:HeckeVectorForm}
        Let $ U_{\mathfrak{t}_\mu, w,r} $ be arbitrary. Let $z\in \mathbb{H} $ and $\zeta \in (1/2, 1] $. Then we define
            \begin{eqnarray*}
                 \vec{F}_{U}(z,\zeta) := \left( p_U^{\vee}\odot V^\vee_U(z) \right)( G_U(z) \odot B_U(\zeta)).
            \end{eqnarray*} 
        to be the \emph{Hecke vector-form} of $ U_{\mathfrak{t}_\mu, w, r} $. 
    \end{defin}

    \begin{notn}\label{not:specialHeckeVectorForm}
        The special case of Definition \ref{def:HeckeVectorForm} that is of greatest interest is
            \begin{eqnarray*}
                \vec{F}_U(z,1) 
                &=& \begin{pmatrix}
                        f_{U,0}(z) & f_{U,1}(z) & \cdots & f_{U,r}(z) 
                    \end{pmatrix}^
                    \text{\emph{T}} \\
                &=& \left( p_U^{\vee}\odot V^\vee_U(z) \right) G_U(z).
            \end{eqnarray*}
        We might say this is the Hecke vector-form in normal position and write
            \begin{eqnarray*}
                \vec{F}_U(z,1) := \vec{F}_U(z)
            \end{eqnarray*}
    \end{notn}

The following result, first derived in \cite{grabner2020} for Hecke vector-forms on $ \mathfrak{H}(1) $, is the culmination of this chapter. 
    \begin{thm}[Fundamental Theorem of Hecke vector-forms]\label{fundamentalHeckeVector-form}
        Per usual, $ T_{\mathfrak{t}_\mu }z $ and $S_{\mathfrak{t}_\mu }z  $ are generators of $ \mathfrak{H}(\varpi_{\mu})$. To recall the matrix definition of $ \mathbf{a}_U $, see Notation \ref{notn:altExchange}. The relations 
            \begin{eqnarray*}
                \vec{F}_{{U}}\left(T_{\mathfrak{t}_\mu}z, \frac{1}{\varpi}\right) = 
                p_U^{\vee}\vec{F}_U(z)
            \end{eqnarray*}
        and
            \begin{eqnarray*}
                 \frac{\vec{F}_{{U}}(S_{\mathfrak{t}_\mu}z)}{z^{w-r}} =  \pm\mathbf{a}_U \vec{F}_{U}(z) 
            \end{eqnarray*}
        hold when we take $  -\mathbf{a}_U $ if $ r $ is odd and $ \mathbf{a}_U $ if $ r $ is even.
    \end{thm}
    \begin{proof}
        We prove the statements in order. Firstly, a short calculation from Defintion \ref{def:HeckeVectorForm} says
            \begin{eqnarray*}
                \vec{F}_U(z,1) = 
                    \begin{pmatrix}
                        f_{U,0}(z) & f_{U,1}(z) & \cdots & f_{U,r}(z) 
                    \end{pmatrix}^{\text{T}}
            \end{eqnarray*}
        is true. Alongside Corollary \ref{cor:auxFUnderT}, observe Proposition \ref{genBinomialTheorem}. Then changing what needs to be changed, we may write
            \begin{eqnarray*}
                &{}& 
                \begin{pmatrix}
                    f_{U,0}(T_{\mathfrak{t}_\mu}z) & f_{U,1}(T_{\mathfrak{t}_\mu}z) & \cdots & f_{U,r}(T_{\mathfrak{t}_\mu}z) 
                \end{pmatrix}^{\text{T}} \\
                &=& p^\vee_U \left\{ \left( p^\vee_U \odot  
                    V^\vee_U(z) \right) (G_U(z) \odot B_U(\varpi) )\right \}
            \end{eqnarray*}
        holds. Therefore,
            \begin{eqnarray*}
                \vec{F}_U(T_{\mathfrak{t}_\mu}z, 1/\varpi) &=& p^\vee _U \left\{ \left( p^\vee_U \odot V^\vee_U(z) \right)\left( ( G_U(z) \odot B_U(\varpi)) \odot  B_U(1/\varpi  )\right)\right \} 
                \\
                &=& p^\vee_U \left\{ \left( p^\vee_U \odot V^\vee_U(z) \right) G_U(z) \right\} \\
                &=& p_U^{\vee}\vec{F}_U(z,1) \\
                &=& p_U^{\vee}\vec{F}_U(z)
            \end{eqnarray*}
        is true, giving us what we wanted per Notation \ref{not:specialHeckeVectorForm}.
        
        The second relation is much simpler. Using Proposition \ref{prop:auxFUnderS} as intermediary reveals
            \begin{eqnarray*}
                \vec{F}_U(S_{\mathfrak{t}_\mu}z,1) &=& z^{w-r}
                    \begin{pmatrix}
                        (-1)^0 f_{U,r}(z) \\
                        (-1)^{1} f_{U,r-1}(z) \\
                        \vdots \\
                        (-1)^r f_{U,0}(z)
                    \end{pmatrix} \\
                    &=& \pm \mathbf{a}_U z^{w-r}\vec{F}_U(z,1)\\
                    &=& \pm \mathbf{a}_U z^{w-r}\vec{F}_U(z)
            \end{eqnarray*}
        holds. Taking $\pm \mathbf{a}_U$ whether $r$ is even or odd as in Example \ref{exa:altExchange}, then we are finished.
    \end{proof}

\section{Hecke automorphic linear differential equations}
The intuition behind automorphic linear differential equations is that we want to find differential equations such that whenever $ y $ is a solution, then $ y(T_{\mathfrak{t}_\mu}z) $ and $ y(S_{\mathfrak{t}_\mu}z) $ are also solutions, $ S_{\mathfrak{t}_\mu}z $ and $ T_{\mathfrak{t}_\mu}z $ being the generators of a Hecke triangle group. This clearly does not guarantee that solutions are automorphic forms -- in fact they often are not. Nevertheless, Hecke vector-forms allow us to give some uniformity to such solutions. 
    \begin{notn}
        Recall that $ \partial^k_w $ is the Ramanujan-Serre derivative. We denote a Hecke automorphic linear differential equation by
            \begin{eqnarray*}
                \mathsf{H}_{\mathbf{B}}[y] :\qquad \partial^{r+1}_{\mathfrak{t}_\mu , w-r}y + B_{\mathfrak{t}_\mu, 4 }\partial^{r-1}_{\mathfrak{t}_\mu , w-r}y + \cdots  + B_{\mathfrak{t}_\mu, 2r+2}y  = 0
            \end{eqnarray*}  
        where $  \mathbf{B} :=(1,0, B_{\mathfrak{t}_\mu, 4 },\cdots, B_{\mathfrak{t}_\mu, 2r+2 })$ and $ B_{\mathfrak{t}_\mu, w } \in \mathcal{M}_{w}(\mathfrak{H}(\varpi_\mu)) $.
    \end{notn}
To solve these equations we have recourse to a simple general method. This method is the so-called Frobenius method. We outline this procedure in the next section.

\subsection{The $q$-Frobenius method}
The Frobenius method (\cite{wangGuo1989}, pp. 61-63) is a classical tool. It is a manner of determining series solutions from $ n $-th order linear ordinary differential equations having regular singularities. This general version is of great significance in our study. For the proof of the Frobenius method for arbitrary $ n \in \mathbb{N} $ see \cite{henrici1977} or \cite{hille1976}. 

Here we explain the method for second-order equations: the crucial ideas are served most clearly in this simplest case. We also introduce some notation that will be used again. Some readers may be familiar with determining series solutions from regular linear differential equations; however, we provide greater detail because we must adapt the classical method to deal with Fourier series.

Let us begin with the simplest relevant ordinary differential equation. Suppose we are considering solutions $ w $ as Fourier series. We use the Boole differential operator $ \theta^1 = \theta := q d / dq $. We know that $ q := e^{2 \pi i z/\varpi} $, but we prefer to work simply with $q$ until it is absolutely necessary to change. Also recall that $\theta^n f = \theta(\theta^{n-1}f ) $. We consider solutions of 
    \begin{eqnarray}\label{classicalSecondOrder}
        \mathsf{L}[w] :\quad \theta^2 w + A_1\theta w + A_0 w = 0,
    \end{eqnarray}
where $ A_1(q) $ and $ A_0(q) $ are some given (coefficient) functions of a complex variable $ e^{2 \pi i z/
\varpi_\mu} = q \in \mathbb{C} $ i.e. they have Fourier expansions. By the analytic theory, we may write a solution $w_1 $ as
    \begin{eqnarray}\label{eqn:regularSolution}
        w_1(q) = q^{\lambda_{1}} \sum^{\infty}_{n=0} w_{1,n} q^n, \label{wSeries} 
    \end{eqnarray}
where the number $ \lambda_1 $ of \eqref{eqn:regularSolution} is the vanishing order of the series $ w_1 $. We will assume the relations
    \begin{eqnarray*}    
        A_1(q) = \sum_{n=0}^{\infty} A_{1,n} q^n , \label{a_1Series} \\
        A_{0}(q) = \sum_{n=0}^{\infty} A_{0,n} q^n \label{a_0Series}
    \end{eqnarray*}
hold. 

    \begin{notn}
        We shall denote the vanishing orders $ \lambda_1,...,\lambda_n $ of respective solutions $ w_1, \cdots, w_n $ of an $ n $-th order differential equation $ \mathsf{L}[w] $ by
            \begin{eqnarray}\label{eqn:vanishingLambda}
                \Lambda_{\mathsf{L}} := (\lambda_1, \cdots, \lambda_n).
            \end{eqnarray}
        As a convention we stipulate that $\lambda_i \geq \lambda_{i + k} $, etc.
    \end{notn}
This set of numbers $ \Lambda_\mathsf{L} $ is the linchpin of the Frobenius method. The reader should convince themselves that the vanishing orders $\Lambda_\mathsf{L} $ depend entirely on \eqref{classicalSecondOrder}. We remark that in our study, we are only interested in the case that the vanishing orders of $ \lambda_\mathsf{L} $ are all integral. 

Returning to the $ n = 2 $ case, our assumptions alongside basic calculations show that the relations
    \begin{eqnarray*}
         w_1 &=& \sum_{n=0}^{\infty} w_{1,n} q^{\lambda_1 + n},  \\
        \theta w_1 &=&  \sum_{n=0}^{\infty}w_{1,n}(\lambda_1 +n)q^{\lambda_1+n}, \\
        \theta^2 w_1 &=& \sum_{n=0}^{\infty}w_{1,n}(\lambda_1 +n)^2 q^{\lambda_1 + n}
    \end{eqnarray*}
hold. Combining in a sum reveals that
    \begin{eqnarray*}
        \theta^2 w_1 +  A_1 \theta w_1 + A_0 w_1   &=& \sum_{n=0}^{\infty}\left[(\lambda_1+n)^2 + (\lambda_1 +n)A_1 + A_0 \right]w_{1,n}q^{\lambda_1 + n}\\  &=& \sum_{n=0}^{\infty} w^*_{1,n}q^{\lambda_1 + n}  
    \end{eqnarray*}
is valid. With that in mind, a simple calculation with the Cauchy product of series gives us valid recursive relations
    \begin{eqnarray*}
        \begin{cases}
            w^*_{1,0} = \left[\lambda_1^2 + \lambda_1 A_{1,0} + A_{0,0}\right] w_{1,0} \\
            w^*_{1,1} = \left[(\lambda_1 +1)^2 + (\lambda_1 +1)A_{1,0}+ A_{0,0}\right] w_{1,1} + \left[(\lambda_1+1)A_{1,1} + A_{0,1}\right] w_{1,0}\\
            \qquad \vdots \\
            w^*_{1,n} = \left[(\lambda_1 +n)^2 + (\lambda_1+n)A_{1,0}+A_{0,0}\right] w_{1,n} \\
            \qquad  \qquad  \qquad \qquad \qquad \qquad \qquad \qquad + \sum_{k=1}^n\left[(\lambda_{1}+n-k) A_{1,k}+A_{0,k} \right]w_{1,n-k}.
        \end{cases}
    \end{eqnarray*}
Setting $ w^*_{1,n} = 0 $, such formulas allow us to determine $ w_{1,n} $ as a rational function in terms of $ \lambda_1 $, $ w_{1,0} $, $ A_{1,k}$, and $ A_{0,k}$, for $ 0 \leq k \leq n $. In other words, letting $ w^*_{1,n} = 0 $, we find
    \begin{eqnarray}\label{eqn:coefficientExpression}
        w_{1,n} = \frac{w_{1,0} \, p(\lambda_1)}{\prod_{k=1}^{n}(\lambda_1+k)^2 +  (\lambda_1+k) A_{1,0} + A_{0,0}} 
    \end{eqnarray}
holds for some polynomial $ p(\lambda_1) $. Notice in (\ref{eqn:coefficientExpression}) that $ \lambda_1 $ can sometimes cause vanishing in the denominator, a problem we must account for.

Isolating from the formula for $ w^*_{1,0} $ in variable $ \lambda $, the polynomial 
    \begin{eqnarray}\label{indicial}
        \text{ind}_{\mathsf{L}}(\lambda) := \lambda^2 +\lambda A_{1,0} + A_{0,0} 
    \end{eqnarray}
is the so-called \emph{indicial polynomial} of $ \mathsf{L}[w] $. In general, one can see that to determine the entire set of fundamental solutions of $ \mathsf{L}[w]$, we must have on hand all solutions of $ \text{ind}_{\mathsf{L}}(\lambda) = 0 $. In other words, the set $\Lambda_{\mathsf{L}} $ of vanishing orders coincides with the roots of the indicial polynomial. It is evident that the order $ n $ of the indicial polynomial is identical to the order of the differential equation $ \mathsf{L}[w] $.

Summarily, that we have shown
    \begin{eqnarray}\label{eqn:DiffEqSeriesForm}
        \mathsf{L}[w_1 ] : \qquad  w_{1,0}\, \text{ind}_{\mathsf{L}}(\lambda_1) q^{\lambda_1}  + q^{\lambda_1} \sum_{n=1}^{\infty} w^*_{1,n}q^{n} = 0
    \end{eqnarray}
holds. Where $ \lambda_1 $ is a root of the indicial polynomial $ \text{ind}_{\mathsf{L}}(\lambda) $, then evidently if $ w^*_{1,n} = 0 $ for $ n = 1, 2, \cdots $ and the series of $ w_1 $ converges, then we will have succeeded in constructing a solution of $\mathsf{L}[w] $.

To give the solutions, let $ \lambda_1 $ and $\lambda_2 $ be the two roots of the equation
    \begin{eqnarray*}
        \text{ind}_{\mathsf{L}}(\lambda) = 0.
    \end{eqnarray*}    
If the difference of $\lambda_1 $ and $ \lambda_2 $ is not an integer, then there is no undesired vanishing in the denominator of (\ref{eqn:coefficientExpression}). The two solutions of equation (\ref{eqn:DiffEqSeriesForm}) are seen from the proceeding to be given by
    \begin{eqnarray*}
        w_{1} = q^{\lambda_1}\sum_{n=0}^{\infty} w_{1,n}q^n , \\
        w_2 = q^{\lambda_2}\sum_{n=0}^{\infty} w_{2,n}q^n.
    \end{eqnarray*}

Next, suppose that $ \text{ind}_{\mathsf{L}}(\lambda) = 0  $ has two solutions $ \lambda_1 = \lambda_2 $. The problem with an integral difference is that (\ref{eqn:coefficientExpression}) becomes undefined as there is vanishing in the denominator. The first solution is as before or 
    \begin{eqnarray*}
        w_1 = q^{\lambda_1}\sum_{n=0}^{\infty}w_{1,n} q^{n}.
    \end{eqnarray*}
The term $ w_{1,0} \text{ind}_{\mathsf{L}}(\lambda_1)z^{\lambda_1} $ in (\ref{eqn:DiffEqSeriesForm}) is where we have a degree of freedom when we construct solutions in the problematic instances. To find the other solution, observe that if $ w^*_{1,n} = 0 $ for $ 1 \leq n $, then it must be the case that in (\ref{eqn:DiffEqSeriesForm}), relation
    \begin{eqnarray*}
        w_{1,0}\, \text{ind}_{\mathsf{L}}(\lambda)q^{\lambda}
        =  w_{1,0}(\lambda -\lambda_1)^2 q^{\lambda}
    \end{eqnarray*}
holds. Next write $ w ' := {d w}/{d\lambda } $. This leaves 
    \begin{eqnarray*}
        \theta^2w'  + a_1 \theta w' + a_0 w'   
        = q^{\lambda}w_{1,0}(\lambda-\lambda_1) \left(2 +(\lambda-\lambda_1)\log q \right) 
    \end{eqnarray*}
holds, an expression that still vanishes at $ \lambda_1 $. Therefore, 
    \begin{eqnarray*}
        w_2 = \left( w'\right)_{\lambda = \lambda_1 }
    \end{eqnarray*}
is the second solution. 

Finally, we assume that $ \lambda_1 - \lambda_2 \in \mathbb{N} $. The first solution is as before or 
    \begin{eqnarray*}
        w_1 = q^{\lambda_1}\sum_{n=0}^{\infty}w_{1,n} q^{n}.
    \end{eqnarray*} 
In the indicial equation (\ref{indicial}) we replace $ w_{1,0} $ with $ \omega (\lambda-\lambda_1 ) $. So that in (\ref{eqn:DiffEqSeriesForm}) we have
    \begin{eqnarray*}
         w_{1,0} \,\text{ind}_{\mathsf{L}}(\lambda)q^{\lambda} = C(\lambda - \lambda_1)^2(\lambda - \lambda_2)q^\lambda
    \end{eqnarray*}
holds where we are free to choose some $ C \neq 0 $. Reasoning as in the previous case, 
    \begin{eqnarray}\label{logAppears}
        w_2 =  w_{1} \log q + \left(  w'\right)_{\lambda = \lambda_1 }
    \end{eqnarray}
is a solution.

We omit the proof that the series $ u_1 $ and $ u_2 $ converge as it is straight-forward (or see \cite{wangGuo1989}).

Continuing our explication, we consider Frobenius applied to the first non-trivial case of a Hecke automorphic linear differential equation, namely where $ r = 1 $ in
    \begin{eqnarray}
            \mathsf{H}[y]: \qquad \partial_{w-1}^{2} y + C E_{\mathfrak{t}_\mu, 4}  y = 0.
    \end{eqnarray}
We make some preparatory calculations. We see that (omitting $C$ for simplicity)
    \begin{eqnarray*}
        \partial_{w-1}^2 y = \theta^2 y - (w-1)\frac{\mu-2}{4\mu}\theta(E_{\mathfrak{t}_\mu, 2}y)   - (w+1)\frac{\mu-2}{4\mu}E_{\mathfrak{t}_\mu, 2}\theta y \qquad \qquad \qquad  \\  + (w+1)(w-r)\left(\frac{\mu-2}{4\mu}\right)^2 E^2_{\mathfrak{t}_\mu, 2}y  \\
        = \theta^2 y  - 2(w+1)\frac{\mu-2}{4\mu}E_{\mathfrak{t}_\mu, 2}\theta y - 
        (w-1)\left( \theta E_{\mathfrak{t}_\mu, 2}  + w\frac{\mu-2}{4\mu}E^2_{\mathfrak{t}_\mu, 2}  \right)\frac{\mu-2}{4\mu} y
    \end{eqnarray*}
holds. We can write this more simply as
    \begin{eqnarray}\label{secondOrderTheta}
        \mathsf{H}_{}[y]:  \qquad \theta^2 y  - 2(w+1)\frac{\mu-2}{4\mu}E_{\mathfrak{t}_\mu, 2}\theta y - 
        (w-1)B_{\mathfrak{t}_\mu, 4}\frac{\mu-2}{4\mu} y = 0,
    \end{eqnarray}
where (here only) we have defined
    \begin{eqnarray*}
         B_{\mathfrak{t}_\mu, 4} &:=& \partial_{w-1}E_{\mathfrak{t}_\mu, 2} + \frac{4 \mu }{(\mu-2)(w-1)}E_{\mathfrak{t}_\mu,4 } \qquad \qquad \qquad \\
         &=& \theta E_{\mathfrak{t}_\mu, 2}  + w\frac{\mu-2}{4\mu}E^2_{\mathfrak{t}_\mu, 2} + \frac{4 \mu }{(\mu-2)(w-1)}E_{\mathfrak{t}_\mu,4 } .
    \end{eqnarray*}
We find that \eqref{secondOrderTheta} is now a familiar second-order equation. 

Just as in the special case found in \cite{grabner2020}, we see the solutions obey
    \begin{eqnarray*}
        y_0 = q^{\lambda_{0}} \sum_{n=0}^{\infty} s_{0,n}q^n
    \end{eqnarray*}
and
    \begin{eqnarray*}
        y_{1} = c_1(1) z y_0 + q^{\lambda_1}\sum_{n=0}^{\infty}s_{1,n}q^n.
    \end{eqnarray*}
For a general equation of degree $ n $, we have a mixed-product identity that we derive now. Notice that 
    \begin{eqnarray*}
        y_\ell(z) = c_\ell(\ell)z^\ell y_0 + c_{\ell -1}(\ell)z^{\ell-1}q^{\lambda_{1}} \sum_{n=0}^{\infty} s_{1,n}q^n + \cdots + q^{\lambda_{r}} \sum_{n=0}^{\infty} s_{r,n}q^n 
    \end{eqnarray*}
holds. Observe that $ 1 = c_0(0) = c_0(1) =  \cdots = c_{0}(r) $
is satisfied, where otherwise the constants $ c_k(\ell) $ are chosen to give the desired recursion. From this it follows 
    \begin{eqnarray*}
        &{}&
        \begin{pmatrix}
            y_0(z) &
            y_1(z) &
            \cdots &
            y_{r}(z)
        \end{pmatrix}^{\text{{T}}}\\
        &=& \left\{
        \begin{pmatrix}
            c_0(0) & &  \\
            c_1(1) & c_0(1) & & \\
            \vdots & \vdots & \ddots & \\
            c_r(r) & c_{r-1}(r) & \cdots & c_0(r)
        \end{pmatrix} \odot
        \begin{pmatrix}
            1 & &  \\
            z & 1 & & \\
            \vdots & \vdots & \ddots & \\
            z^r & z^{r-1} & \cdots & 1
        \end{pmatrix} \right\}\\
        &\times&
        \begin{pmatrix}
            q^{\lambda_0}\sum_{n=0}^\infty s_{0,n}q^n \\
            q^{\lambda_1}\sum_{n=0}^\infty s_{1,n}q^n \\
            \\
            q^{\lambda_r}\sum_{n=0}^\infty s_{r,n}q^n
        \end{pmatrix}\\
        &=& (C^{\vee}_{\mathsf{H}} \odot V^{\vee}_{\mathsf{H}}(z)) S_{\mathsf{H}}(z)
    \end{eqnarray*}
holds (notation being obvious).  Abbreviating our ansatz then, we have that the mixed-product identity
    \begin{eqnarray}\label{ansatzVector}
        \vec{Y}_{\mathsf{H}}(z) = (C^{\vee}_{\mathsf{H}} \odot V^{\vee}_{\mathsf{H}}(z)) S_{\mathsf{H}}(z)
    \end{eqnarray}
holds for $ \vec{Y}_{\mathsf{H}} $ a fundamental solution vector of $\mathsf{H}_\mathbf{B}$. 

    \begin{rmk}
        We remark that the matrix $ V_U^\vee(z) $ appears in \eqref{ansatzVector} as a result of $ \log q $ where $ {q := e^{2 \pi i z/\varpi_\mu}} $, as can be seen e.g. at equation \eqref{logAppears}. Furthermore, it is worthwhile to compare equation \eqref{ansatzVector} with the form of the mixed-product binomial theorem of Proposition \ref{genBinomialTheorem}.
    \end{rmk}

This completes our illustration of the Frobenius method.

\subsection{The adapted Wronskian}
The Wronskian is another classical notion, but we want to consider a Wronskian with a non-standard differential operator. With respect to quasiautomorphic forms, the most natural differential operator is the Ramanujan-Serre operator $ \partial_{\mathfrak{t}_\mu, w} $ (see Definition \ref{Ramanujan-Serre}) as it exhibits closure with respect to quasiautomorphic forms. Due to this alteration, we will adapt the classical Wronskian in a
    \begin{defin}
        Let $ \vec{F}_{U}(z,1) = \vec{F}_{U}(z)  $ be the Hecke vector-form of $ U_{\mathfrak{t}_\mu, w,r} $ and $ \partial_{w-r}^k $ be the Ramanujan-Serre differential operator. Then we define the \emph{Wronskian} by
            \begin{eqnarray*}
                \mathcal{W}_{\vec{F}}(z) &:=& \det
                    \begin{pmatrix}
                        f_{U, 0}(z) & 
                        \partial_{w-r}f_{U, 0}(z) &
                        \cdots &
                        \partial_{w-r}^r f_{U, 0}(z) \\
                        f_{U, 1}(z) & 
                        \partial_{w-r}f_{U, 1}(z) &
                        \cdots &
                        \partial_{w-r}^r f_{U, 1}(z) \\
                        \vdots & \vdots & \ddots & \vdots \\
                        f_{U, r}(z) & 
                        \partial_{w-r}f_{U, r}(z) &
                        \cdots & 
                        \partial_{w-r}^r f_{U, r}(z) 
                    \end{pmatrix}
                    \\ &=&  \det
                    \begin{pmatrix}
                        \vec{F}_U(z) & 
                        \partial_{w-r}\vec{F}_U(z) &
                        \cdots &
                        \partial_{w-r}^r \vec{F}_U(z)
                    \end{pmatrix}.
            \end{eqnarray*}
    \end{defin}

    \begin{rmk}
        With respect to our choice of using the Ramanujan-Serre derivative, we find that the closure relation
            \begin{eqnarray*}
                \partial_{w-r} U_{\mathfrak{t}_\mu, w, r} : \mathcal{QM}^r_w(\mathfrak{H}(\varpi_\mu))\rightarrow \mathcal{QM}^r_{w+2}(\mathfrak{H}(\varpi_\mu))  
            \end{eqnarray*}
        holds. This is the reason for indexing the Ramanujan-Serre derivative with $ w-r $.      
    \end{rmk}

It turns out that our adaptation of the Wronskian does not change things very dramatically. The following proposition clarifies this assertion. 
    \begin{prop}\label{WronskiSimilarityInvariant}
        Let $ \vec{F}_{U}(z,1) = \vec{F}_{U}(z) $ be the Hecke vector-form of some $ U_{\mathfrak{t}_\mu, w,r} $. As before $ {\theta  : = q d/dq} $ is the Boole differential operator. The relation
            \begin{eqnarray*}
                \mathcal{W}_{\vec{F}} = \det
                    \begin{pmatrix}
                        \vec{F}_U & 
                        \theta \vec{F}_U &
                        \cdots &
                        \theta ^r \vec{F}_U
                    \end{pmatrix}.
            \end{eqnarray*}
        holds.
    \end{prop}
    \begin{proof}
         We use the following property of determinants: adding a scalar multiple of one column of a matrix to another column leaves the determinant invariant. Thus, we need to show that elements of the square matrix
            \begin{eqnarray*}
                \begin{pmatrix}
                    \vec{F}_U & 
                    \partial_{w-r} \vec{F}_U &
                    \cdots &
                    \partial_{w-r}^r \vec{F}_U
                \end{pmatrix}
            \end{eqnarray*}
        are none other than scalar multiples of terms $ \vec{F}_U $, $ \theta\vec{F}_U $, $\cdots $, $ \theta^r \vec{F}_U $. For the base case, observe that
            \begin{eqnarray*}
                \partial \vec{F}_U =  \theta \vec{F}_U - aE_2 \vec{F}_U 
            \end{eqnarray*}
        holds. It is clear that $\theta \vec{F}_U $ and $ - aE_2 \vec{F}_U $ are scalar multiples of $ \theta \vec{F}_U $ and $ \vec{F}_U $, respectively. 

        Next, suppose that 
            \begin{eqnarray*}
                \mathcal{W}_{\vec{F}} =
                    \begin{pmatrix}
                        \vec{F}_U & 
                        \partial_{w-r} \vec{F}_U &
                        \cdots &
                        \partial_{w-r}^r \vec{F}_U
                    \end{pmatrix}
            \end{eqnarray*}
        satisfies the hypothesis. Then we find
            \begin{eqnarray*}
                    \begin{pmatrix}
                        \vec{F}_U & 
                        \partial_{w-r} \vec{F}_U &
                        \cdots &
                        \partial_{w-r}^{r+1} \vec{F}_U
                    \end{pmatrix} \qquad \qquad \qquad \qquad \quad\\
                =
                    \begin{pmatrix}
                        \vec{F}_U & 
                        \partial_{w-r} \vec{F}_U &
                        \cdots &
                        \partial_{w-r}^r \vec{F}_U &
                        \theta (\partial_{w-r}^{r} \vec{F}_U) - bE_{\mathfrak{t}_\mu, 2} (\partial_{w-r}^{r} \vec{F}_U)
                    \end{pmatrix}
            \end{eqnarray*}
        holds. By inspection we only need to show $ \theta (\partial_{w-r}^{r} \vec{F}_U) $ is a linear combination of the elements $ \vec{F}_U $, $ \theta \vec{F}_U $, $\cdots$ $ \theta^{r+1} \vec{F}_U $; however, this holds by linearity of the differential operator $ \theta $ and the induction hypothesis. The proof is complete.
    \end{proof}

We shall use the Wronskian again later, but first we demonstrate its use in an instructive, simple result. 
    \begin{prop}\label{prop:WronskiIsModular}
        Let $ U_{\mathfrak{t}_\mu, w, r} $ be arbitrary. If $\vec{F}_U $ is the Hecke vector-form of $ U_{\mathfrak{t}_\mu, w, r}  $, then the relation 
            \begin{eqnarray*}
                \mathcal{W}_{\vec{F}} \in \mathcal{M}_{w(r+1)}(\mathfrak{H}(\varpi_\mu))
            \end{eqnarray*}
        holds. 
    \end{prop}
    \begin{proof}
        We show the effect of $T_{\mathfrak{t}_\mu z}$ and $ S_{\mathfrak{t}_\mu z} $ on $ \mathcal{W}_{\vec{F}} $. Recall that linear alterations to columns of a matrix do not change the determinant. 

        By the first case of the Fundamental Theorem 
            \begin{eqnarray*}
                \vec{F}_{{U}}\left(T_{\mathfrak{t}_\mu}z, \frac{1}{\varpi}\right) = 
                p_U^{\vee}\vec{F}_U(z,1) = p_U^{\vee}\vec{F}_{U}(z)
            \end{eqnarray*}
        holds. Substitution reveals
            \begin{eqnarray*}
                \mathcal{W}_{\vec{F}}(T_{\mathfrak{t}_\mu z}) =
                    \begin{pmatrix}
                        p_U^{\vee}\vec{F}_U(z) & 
                        \partial_{w-r} p_U^{\vee}\vec{F}_U(z) &
                        \cdots &
                        \partial_{w-r}^r p_U^{\vee}\vec{F}_U(z)
                    \end{pmatrix}
            \end{eqnarray*}
        holds. Though we may use linearity of differential operator to simplify further i.e. 
            \begin{eqnarray*}
                \partial_{w-r}^k p_U^{\vee}\vec{F}_U(z) = p_U^{\vee} \partial_{w-r}^k \vec{F}_U(z) 
            \end{eqnarray*}
        holds. We thus see that since the determinant of a triangular matrix is the product of the diagonal elements, here all being unity, it follows that $\det p_U^\vee = 1$. We are then permitted to write $ \mathcal{W}_{\vec{F}}(T_{\mathfrak{t}_\mu}z) = \mathcal{W}_{\vec{F}}(z) $ holds, just as we wish.

        By the second case of the Fundamental theorem 
            \begin{eqnarray*}
                 {\vec{F}_{{U}}(S_{\mathfrak{t}_\mu}z,1)} =  \pm  {z^{w-r}} \mathbf{a}_U \vec{F}_{U}(z,1) = \pm  {z^{w-r}} \mathbf{a}_U \vec{F}_{U}(z) 
            \end{eqnarray*}
            holds when we take $  1 $ if $ r $ is even and $ - 1 $ if $ r $ is odd. By the previous case, it is enough to observe $ \pm \det (\mathbf{a}) = 1 $. In other words, 
                \begin{eqnarray*}
                    \frac{\mathcal{W}_{\vec{F}}(S_{\mathfrak{t}_\mu}z)}{z^{w(r+1)}} = \mathcal{W}_{\vec{F}}(z)
                \end{eqnarray*}
            holds i.e. the Wronskian here is an automorphic form of weight $ w(r+1) $. This completes the proof.
    \end{proof}

\subsection{Fundamental solutions of Hecke automorphic linear ODEs}
Again, we denote a Hecke automorphic linear differential equation by
            \begin{eqnarray}\label{HeckeAutomorphicDiffEq}
                \mathsf{H}_{\mathbf{B}}[y] :\qquad \partial^{r+1}_{\mathfrak{t}_\mu , w-r}y + B_{\mathfrak{t}_\mu, 4 }\partial^{r-1}_{\mathfrak{t}_\mu , w-r}y + \cdots  + B_{\mathfrak{t}_\mu, 2r+2}y  = 0
            \end{eqnarray}  
        where $  \mathbf{B} :=(1,0, B_{\mathfrak{t}_\mu, 4 },\cdots, B_{\mathfrak{t}_\mu, 2r+2 })$ and $ B_{\mathfrak{t}_\mu, w } \in \mathcal{M}_{w}(\mathfrak{H}(\varpi_\mu)) $. Suppose that 
    \begin{eqnarray*}
        \vec{Y}_{\mathsf{H}}(z) := 
            \begin{pmatrix}
                y_{0}(z) &
                y_{1}(z) &
                \cdots &
                y_{r}(z)
            \end{pmatrix}^{\text{T}}
    \end{eqnarray*}
is the fundamental solution vector of equation $ \mathsf{H}_{\mathbf{B}}[y] $.

    \begin{thm}\label{thm:fundSolutions}
        With respect to $ \mathsf{H}_{\mathbf{B}}[y] $ suppose that $ \Lambda_\mathsf{H} = (\lambda_0, \cdots, \lambda_r )$ is a set of integers satisfying $ \lambda_0 \geq \lambda_1 \geq \cdots \geq \lambda_r \geq 0 $. Then the relation 
            \begin{eqnarray*}
                \mathcal{W}_{\vec{Y}} = C \Delta_{\mathfrak{t}_\mu}^{\frac{w(r+1)}{\delta_{\mathfrak{t}_\mu}}}
            \end{eqnarray*}
        holds for some complex $ C \neq 0 $. 
    \end{thm}
    \begin{proof}
        This proof is identical to that of the proof of Proposition 3.4 in \cite{grabner2020}, except for two small matters.
        
        Firstly, we must substitute  $ \delta_{\mathfrak{t}_\mu} = 2 \lcm(2, \mu) $ where there would be a $ 12 $ in the source. This causes no trouble. 

        Secondly, under $ T_{\mathfrak{t}_\mu}z $ we must instead consider the ansatz vector
            \begin{eqnarray*}
                \vec{Y}_{\mathsf{H}}(T_{\mathfrak{t}_\mu}z) = (C_\mathsf{H}^\vee \odot V_{\mathsf{H}}^\vee (T_{\mathfrak{t}_\mu}z) ) S_\mathsf{H}(T_{\mathfrak{t}_\mu}z).
            \end{eqnarray*}
        from \eqref{ansatzVector}. Since $ q := e^{2 \pi i z/\varpi} $, we have 
            \begin{eqnarray*}
                 S_\mathsf{H}(T_{\mathfrak{t}_\mu}z) = S_\mathsf{H}(z)
            \end{eqnarray*}
        holds. But $ V_{\mathsf{H}}^\vee (T_{\mathfrak{t}_\mu}z) = V^\vee_\mathsf{H}(z+\varpi) $ where possibly $ \varpi \neq 1$; however, this does not change the determinant of $ \vec{Y}_{\mathsf{H}}(z) $ by properties of the determinant. In other words,
            \begin{eqnarray*}
                \mathcal{W}_{\vec{Y}}(T_{\mathfrak{t}_\mu}z ) = \rho(T_{\mathfrak{t}_\mu})\mathcal{W}_{\vec{Y}}(z) ,
            \end{eqnarray*}
        where $ \rho(T_{\mathfrak{t}_\mu}) $ is analogous to that of the proof found in \cite{grabner2020}. All else being equal, the proof is complete. 
    \end{proof}

    \begin{exa}
        There is a famous identity of Garvan, developed further in \cite{milne2001}, that says,
            \begin{eqnarray*}
                \det
                    \begin{pmatrix}
                        E_{\mathfrak{t}_3, 4} & E_{\mathfrak{t}_3, 6} & E_{\mathfrak{t}_3, 8} \\
                        E_{\mathfrak{t}_3, 6} & E_{\mathfrak{t}_3, 8} & E_{\mathfrak{t}_3, 10} \\
                        E_{\mathfrak{t}_3, 8} & E_{\mathfrak{t}_3, 10} & E_{\mathfrak{t}_3, 12}\\
                    \end{pmatrix}
                = -\frac{250\,(1728)^2}{691 (2\pi)^{24}}\Delta_{\mathfrak{t}_3}^2
            \end{eqnarray*}
        holds. By Theorem \ref{thm:fundSolutions} it would appear that the Hankel matrix property (Definition \ref{def:toeplitz}) is not fundamental to such identities.
    \end{exa}

In Theorem \ref{thm:fundSolutions}, we have shown that the Wronskian of fundamental solutions is a predictable power of the discriminant function times some scalar. We next prove a Proposition that will allow us to relate an arbitrary Hecke vector-form to the fundamental solutions of equation \eqref{HeckeAutomorphicDiffEq}.

    \begin{prop}\label{vectorFormSolutionVector}
        Let $ U_{\mathfrak{t}_\mu, w, r} \in \mathcal{QM}^r_w(\mathfrak{H}(\varpi_\mu)) $ and suppose $ U_{\mathfrak{t}_\mu, w, r} $ is a solution to some $ \mathsf{H}_\mathbf{B}[y] $ as above. Then $ \vec{F}_U(z,1) $ is a fundamental solution vector of $ \mathsf{H}_\mathbf{B}[y] $. 
    \end{prop}
    \begin{proof}
        Recall that if $ y $ is a solution of a Hecke linear automorphic differential equation, then also $ y(S_{\mathfrak{t}_\mu}z) $ is a solution. We prove the desired claim as a corollary of an identity shown in \cite{grabner2020}. First, choose $ f_{U, r}(z) $ as it is given in Definition \ref{def:components}; clearly, $ r $ needs to be the depth of the solution $ y $ and we are working over $ \mathfrak{H}(1) $ or the full modular group. Then in \cite{grabner2020}, they show the identity
            \begin{eqnarray*}
                \begin{pmatrix}
                    f_{U, r}(z) \\
                    f_{U,r}(T_{\mathfrak{t}_3}z) \\
                    \vdots \\
                    f_{U,r} (T^r_{\mathfrak{t}_3}z)
                \end{pmatrix} = 
                \mathcal{V}_U(\overline{r}) \left\{ 
                    \mathbf{b}_U \odot 
                    \begin{pmatrix}
                        f_{U,0}(z) \\
                        f_{U,1}(z) \\
                        \vdots \\
                        f_{U,r}(z)
                    \end{pmatrix}\right\}
            \end{eqnarray*}
        holds if 
            \begin{eqnarray*}
                \mathcal{V}_U(\overline{r}) := \mathcal{V}_U(0,1,\cdots, r)
            \end{eqnarray*}
        (recall that $\mathcal{V}_U $ is the Vandermonde matrix of Definition \ref{def:vandermondeMatrix}). By Proposition \ref{prop:fWithWeights}, the left side is equivalent to relations
            \begin{eqnarray*}
                \begin{pmatrix}
                    f_{U, r}(z) \\
                    f_{U, r}(T_{\mathfrak{t}_3}z) \\
                    \vdots \\
                    f_{U, r}(T^r_{\mathfrak{t}_3}z)
                \end{pmatrix} &=& 
                \mathcal{V}_U(\overline{r}) \left\{ 
                    \mathbf{b}_U \odot \left[
                    (p^{\vee}_U\odot V^{\vee}_U(z))(  G_{U}(z) \odot B_U(1) ) \right] \right \} \\
                    &=& \mathcal{V}_U(\overline{r}) \left\{ 
                    \mathbf{b}_U \odot \vec{F}_U(z,1) \right\}\\
                    &=& \mathcal{V}_U(\overline{r}) \left\{ 
                    \mathbf{b}_U \odot \vec{F}_U(z) \right\}.
            \end{eqnarray*}
        
        We can move to quasiautomorphic forms over $\mathfrak{H}(\varpi_\mu) $, after changing what needs to be changed in $ U_{\mathfrak{t}_\mu} $, by seeing that 
            \begin{eqnarray*}
                \begin{pmatrix}
                    f_{U, r}(z) \\
                    f_{U, r}(T_{\mathfrak{t}_\mu}z) \\
                    \vdots \\
                    f_{U, r}(T^r_{\mathfrak{t}_\mu}z)
                \end{pmatrix} &=& 
                    \mathcal{V}_U(\overline{r}) \left\{ 
                    \mathbf{b}_U \odot 
                    \left[(p^{\vee}_U\odot V^{\vee}_U(z))(G_{U}(z) \odot B_U(\varpi) ) \right] \right \} \\ &=& 
                    \mathcal{V}_U(\overline{r}) \left\{ 
                    \mathbf{b}_U \odot 
                    \vec{F}_U(z,\varpi)  \right\}
            \end{eqnarray*}
        holds. The operation $ B_U(\varpi) \odot G_{U}(z) $ is linear in that we multiply each column element of $ G_{U}(z) $ by a scalar. By invariance of the Wronski under linear changes in matrix columns we conclude that $ \vec{F}_U(z,1) = \vec{F}_U(z) $ is a fundamental solution vector of $\mathsf{H}_{\mathbf{B}}$.
    \end{proof}

\printbibliography
    
\end{document}